\documentclass[10pt]{amsart}
\usepackage{amsmath,amssymb,amsthm,graphicx,epstopdf,mathrsfs,url}
\usepackage[usenames,dvipsnames]{color}
\usepackage{dsfont}   
\definecolor{darkred}{rgb}{0.4,0.1,0.1}
\definecolor{darkblue}{rgb}{0.1,0.1,0.4}
\definecolor{darkgrey}{rgb}{0.5,0.5,0.5}
\usepackage[colorlinks=true,linkcolor=darkred,citecolor=Blue]{hyperref}

\numberwithin{equation}{section}
\theoremstyle{plain}
\newtheorem{thm}{Theorem}[section]
\newtheorem{lem}[thm]{Lemma}

\newtheorem{prop}[thm]{Proposition}
\newtheorem{cor}[thm]{Corollary}

\newtheorem{definition}[thm]{Definition}
\theoremstyle{remark}
\newtheorem{remark}[thm]{Remark}

\theoremstyle{plain}

\newcommand{\hyp}[1]{$C^{2}$-hypersurface as in Definition~\ref{definition_hypersurface}}

\DeclareMathOperator\ran{ran}

\newcommand{\dom}{\mathrm{dom}\,}

\begin{document}
\title[]{On Dirac operators with electrostatic $\delta$-shell interactions of critical strength}
\author{Jussi Behrndt}
\address{Institut f\"{u}r Numerische Mathematik\\
Technische Universit\"{a}t Graz\\
 Steyrergasse 30, 8010 Graz, Austria\\
E-mail: {\tt behrndt@tugraz.at}}

\author{Markus Holzmann}
\address{Institut f\"{u}r Numerische Mathematik\\
Technische Universit\"{a}t Graz\\
 Steyrergasse 30, 8010 Graz, Austria\\
E-mail: {\tt holzmann@math.tugraz.at}}

\begin{abstract}
  In this paper we prove that the Dirac operator $A_\eta$ with an electrostatic $\delta$-shell interaction
of critical strength $\eta = \pm 2$ supported on a $C^2$-smooth compact surface $\Sigma$
is self-adjoint in $L^2(\mathbb{R}^3;\mathbb{C}^4)$, we 
describe the domain explicitly in terms of traces and jump conditions in $H^{-1/2}(\Sigma; \mathbb{C}^4)$,
and we investigate the spectral properties of $A_\eta$.
While the non-critical interaction strengths $\eta \not= \pm 2$ have received a lot of attention in the recent past,
the critical case $\eta = \pm 2$ remained open. Our approach is based on abstract techniques in 
extension theory of symmetric operators,
in particular, boundary triples and their Weyl functions. 
\end{abstract}


\maketitle

\section{Introduction} \label{section_introduction}

Dirac operators with electrostatic $\delta$-shell interactions attracted a lot of attention in the recent past,
see~\cite{AMV14, AMV15, AMV16, BEHL16_2, M15, MP16, OV16} or the related papers~\cite{ALTR16, BFSB17_1, BFSB17_2}.
From the physical point of view they are the relativistic counterpart of Schr\"odinger operators with 
$\delta$-potentials, which are used as idealized models for Schr\"odinger operators
with strongly localized regular potentials, cf.~\cite{AGHH05, BEHL16_1, E08, KP31} and the references therein.
On the other hand Dirac operators with 
electrostatic $\delta$-shell interactions are also interesting from the mathematical point of view, since
it can be expected that their spectral properties depend on the geometry of the interaction support and/or the 
interaction strength; such effects are studied in the monograph \cite{EK15}
and, e.g., in \cite{BEKS94, E05, E08, EHL06, EP14}
for Schr\"odinger operators with $\delta$-potentials.

The mathematical study of Dirac operators with singular interactions supported on a set of measure zero started in 
the 1980s. In the one-dimensional case several results as, e.g., a description of the spectrum, an explicit resolvent 
formula, the approximation by Dirac operators with squeezed potentials and their convergence in the nonrelativistic 
limit were deduced in \cite{AGHH05, CMP13,
GS87, PR14, S89}. Making use of these results and a decomposition into spherical harmonics 
J.~Dittrich, P.~Exner, and P.~\v{S}eba studied
the Dirac operator in $\mathbb{R}^3$ with a singular perturbation supported on a sphere in \cite{DES89}.
The investigation of the Dirac operator in $\mathbb{R}^3$ with singular perturbations supported on more general surfaces 
was initiated only recently in the pioneering paper~\cite{AMV14} by N.~Arrizabalaga, A.~Mas, and L.~Vega,
where a new approach to extension theory of symmetric operators was employed; this research was continued in 
\cite{AMV15, AMV16, M15, MP16}. 
A different approach using the abstract theory of quasi boundary triples and their Weyl functions from~\cite{BL07, BL12}
was proposed by P.~Exner, V.~Lotoreichik, and the authors of the present paper in \cite{BEHL16_2}.

In what follows we fix some notations and describe several already obtained results to set up the problem treated in this paper.
Let us choose units such that the Planck constant $\hbar$ and the speed of light are both equal to one. 
The free Dirac operator $A_0$ in 
$L^2(\mathbb{R}^3; \mathbb{C}^4)$ is given by 
\begin{equation*} 
  A_0 f := -i \sum_{j=1}^3 \alpha_j \partial_j f + m \beta f, \qquad \dom A_0 = H^1(\mathbb{R}^3; \mathbb{C}^4),
\end{equation*}
where the Dirac matrices $\alpha_1, \alpha_2, \alpha_3$ and $\beta$ are defined by \eqref{def_Dirac_matrices} below.
The operator $A_0$ describes the motion of a free spin-$\frac{1}{2}$ particle with mass $m>0$ in $\mathbb{R}^3$ taking relativistic effects
into account. Furthermore, let $\Sigma$ be the boundary of a bounded 
$C^2$-smooth domain $\Omega_+ \subset \mathbb{R}^3$ and let $\Omega_- := \mathbb{R}^3 \setminus \overline{\Omega_+}$.
The Dirac operator with an electrostatic $\delta$-shell interaction of strength $\eta \in \mathbb{R}$ supported 
on $\Sigma$ is formally given by
\begin{equation*}
  A_\eta = A_0 + \eta I_4 \delta_\Sigma;
\end{equation*}
here $I_4$ stands for the identity matrix in $\mathbb{C}^{4 \times 4}$. In a mathematically rigorous form $A_\eta$, $\eta\not=\pm 2$, is defined
in \cite{AMV14, BEHL16_2} as a particular self-adjoint extension of the symmetric operator 
\begin{equation*}
  S := A_0 \upharpoonright H^1_0(\mathbb{R}^3 \setminus \Sigma; \mathbb{C}^4).
\end{equation*}
Observe that $S$ is the restriction of the free Dirac operator to functions that vanish on $\Sigma$.
Roughly speaking, a function $f \in \dom S^*$ belongs to $\dom A_\eta$ if the traces of 
$f_\pm := f\upharpoonright \Omega_\pm$ satisfy the jump condition
\begin{equation} \label{equation_jump_condition_intro}
  \frac{\eta}{2} \big( f_+|_\Sigma + f_-|_\Sigma \big) = -i \alpha \cdot \nu \big( f_+|_\Sigma - f_-|_\Sigma \big),
\end{equation}
where $\nu$ is the outer unit normal vector field of $\Omega_+$; cf. Definition~\ref{definition_A_eta} for more details.
Concerning the basic spectral properties of $A_\eta$ in the non-critical case $\eta\not=\pm 2$ the following theorem is known from 
\cite{AMV14, BEHL16_2} and Proposition~\ref{proposition_A_eta_self_adjoint_nice}.

\begin{thm} \label{theorem_spectrum_noncritical}
  For $\eta \in \mathbb{R} \setminus \{ \pm 2\}$ the operator $A_\eta$ is self-adjoint in $L^2(\mathbb{R}^3; \mathbb{C}^4)$ 
  and the following properties hold:
  \begin{itemize}
    \item[\rm (i)] $\dom A_\eta \subset H^1(\mathbb{R}^3 \setminus \Sigma; \mathbb{C}^4)$;
    \item[\rm (ii)] the essential spectrum of $A_\eta$ is given by
    \begin{equation*}
      \sigma_{\mathrm{ess}}(A_\eta) = (-\infty, -m] \cup [m, \infty);
    \end{equation*}
    \item[\rm (iii)] the discrete spectrum of $A_\eta$ in the gap $(-m, m)$ is finite.
  \end{itemize}
\end{thm}

For an interaction strength $\eta \in \mathbb{R} \setminus \{ \pm 2 \}$ also 
various other results for the operator $A_\eta$ are known,
as, e.g., an abstract version of the Birman-Schwinger 
principle~\cite{AMV15, BEHL16_2},
an isoperimetric inequality~\cite{AMV16}, the existence and completeness of the 
wave operators for the pair $\{ A_\eta, A_0 \}$, the convergence in the nonrelativistic
limit~\cite{BEHL16_2}, and the approximation by Dirac operators with squeezed potentials including Klein's paradox \cite{MP16}.

We emphasize that in all papers \cite{AMV14, AMV15, AMV16, BEHL16_2, M15, MP16}
the critical interaction strengths $\eta = \pm 2$ were excluded. 
This situation is more difficult to handle with extension theoretic techniques and remained open so far. 
It is the goal of this paper to fill this gap. Our main result can be summarized as follows; cf. 
Theorem~\ref{theorem_self_adjoint_bad}, Theorem~\ref{proposition_properties_A_pm2}, and 
Theorem~\ref{theorem_essential_spectrum} for more details. 

\begin{thm} \label{theorem_main_result}
  The Dirac operator with an electrostatic $\delta$-shell interaction of critical strength $\eta = \pm 2$
  is self-adjoint in $L^2(\mathbb{R}^3; \mathbb{C}^4)$, its domain is not contained in 
  $H^1(\mathbb{R}^3 \setminus \Sigma; \mathbb{C}^4)$, the set $(-\infty,-m]\cup[m,\infty)$ belongs to the essential spectrum and essential spectrum may
  also appear in $(-m,m)$.
\end{thm}

In fact, it will turn out in 
Theorem~\ref{theorem_self_adjoint_bad} that the operator $A_{\pm 2}$ is essentially self-adjoint in $L^2(\mathbb{R}^3; \mathbb{C}^4)$
and hence, its closure is self-adjoint. Here $A_{\pm 2}$ is defined with the help of a suitable quasi boundary triple in a similar way as in \cite{BEHL16_2} 
on functions satisfying the jump condition \eqref{equation_jump_condition_intro}
in~$H^{1/2}(\Sigma; \mathbb{C}^4)$.
Our techniques, based on special 
transformations of quasi boundary triples to ordinary boundary triples and vice versa in the spirit of \cite{BM14}, 
allow us to give an explicit description of the domain of the self-adjoint operator 
$\overline{A_{\pm 2}}$. More precisely, we show that $f \in \dom S^*$ belongs to
the domain 
of the self-adjoint Dirac operator $\overline{A_{\pm 2}}$ with critical interaction strength 
if and only if the traces of $f$ satisfy the jump condition 
in~\eqref{equation_jump_condition_intro} in $H^{-1/2}(\Sigma; \mathbb{C}^4)$ 
and that $\dom \overline{A_{\pm 2}}$
is not contained in  $H^1(\mathbb{R}^3 \setminus \Sigma; \mathbb{C}^4)$.
Thus the functions in $\dom\overline{A_{\pm 2}}$ are less regular
than those in $\dom A_\eta$, $\eta\in\mathbb{R} \setminus \{ \pm 2\}$, which
indicates one of the key difficulties in the treatment of the critical interaction strengths $\pm 2$. 
We would like to point out that a result of the same type as Theorem~\ref{theorem_self_adjoint_bad} 
was obtained recently in \cite{OV16} by T.~Ourmi{\`e}res-Bonafos and L.~Vega. In the present paper
we also investigate the spectral properties of the 
self-adjoint operators $\overline{A_{\pm 2}}$. As one may expect the set $(-\infty,-m]\cup[m,\infty)$
belongs to the essential spectrum -- the proof of this fact
is based on the usage of suitable singular sequences -- but it is less intuitive that also in the interval $(-m,m)$ 
essential spectrum may appear. For the case that the interaction support $\Sigma$
contains a flat part we prove in Theorem~\ref{theorem_essential_spectrum} that the point~$0$
belongs to $\sigma_{\rm ess}(\overline{A_{\pm 2}})$ and at the same time it turns
out that in this situation the functions in $\dom {\overline{A_{\pm 2}}}$ do not possess any
Sobolev regularity of positive order.
We remark that a similar effect occurs in the study of indefinite Laplacians; 
cf.~\cite{BK17, CPP17}.

The paper is organized as follows: In Section~\ref{section_boundary_triples} we provide some 
statements from the theory of quasi and ordinary boundary triples that are needed to prove our main results. 
Section~\ref{section_free_Dirac} contains then some preliminary considerations on the free Dirac operator 
in $\mathbb{R}^3$ and a maximal Dirac operator in~$\mathbb{R}^3 \setminus \Sigma$, 
while in Section~\ref{section_boundary_triples_Dirac} boundary triples suitable for
Dirac operators with singular interactions are studied. 
Section~\ref{section_Dirac_delta} contains our main results: Theorem~\ref{theorem_self_adjoint_bad}, Theorem~\ref{proposition_properties_A_pm2}, and 
Theorem~\ref{theorem_essential_spectrum}.

\subsection*{Notations} \label{section_notations}
The positive constant $m$ stands for the mass of the particle. 
The identity matrix in $\mathbb{C}^{n \times n}$ is denoted by $I_n$.
Furthermore, $\alpha_1, \alpha_2, \alpha_3$ and $\beta$ are the Dirac matrices 
\begin{equation} \label{def_Dirac_matrices}
  \alpha_j := \begin{pmatrix} 0 & \sigma_j \\ \sigma_j & 0 \end{pmatrix} 
  \quad \text{and} \quad \beta := \begin{pmatrix} I_2 & 0 \\ 0 & -I_2 \end{pmatrix},
\end{equation}
where $\sigma_j$ are the Pauli spin matrices 
\begin{equation*}
  \sigma_1 := \begin{pmatrix} 0 & 1 \\ 1 & 0 \end{pmatrix}, \qquad
  \sigma_2 := \begin{pmatrix} 0 & -i \\ i & 0 \end{pmatrix}, \qquad
  \sigma_3 := \begin{pmatrix} 1 & 0 \\ 0 & -1 \end{pmatrix}.
\end{equation*}
The Dirac matrices satisfy the anti-commutation relations
\begin{equation} \label{equation_anti_commutation}
  \alpha_j \alpha_k + \alpha_k \alpha_j = 2 \delta_{j k} \quad \text{and} \quad \alpha_j \beta + \beta \alpha_j = 0,
  \quad j, k \in \{ 1, 2, 3 \}.
\end{equation}
For vectors $x = (x_1, x_2, x_3)^{\top}$ we employ the notation
$\alpha \cdot x := \sum_{j=1}^3 \alpha_j x_j$.

The open ball of radius $R$ centered at $x$ is denoted by $B(x,R)$. 
Moreover, $\Omega_+ \subset \mathbb{R}^3$ is a $C^2$-smooth bounded domain 
and we set $\Omega_- := \mathbb{R}^3 \setminus \overline{\Omega_+}$ and $\Sigma := \partial \Omega_+$.
For an open set $\Omega \subset \mathbb{R}^3$ we write 
$C^\infty_c(\Omega; \mathbb{C}^4)$ for the space of all infinitely many times differentiable vector valued functions 
with four components and compact support, and 
$C^\infty(\overline{\Omega}; \mathbb{C}^4) := \{ f\upharpoonright \Omega: f \in C^\infty_c(\mathbb{R}^3; \mathbb{C}^4) \}$.
In a similar way, if $\Omega$ is an open subset of $\mathbb{R}^3$ or if $\Omega = \Sigma$, then
$L^2(\Omega; \mathbb{C}^4)$ denotes the space of vector valued functions, where each of the four components 
is square integrable, and we write $(\cdot, \cdot)_\Omega$ for the corresponding inner product. If $\Omega = \Sigma$,
then these $L^2$-spaces are equipped with the Hausdorff measure $\sigma$, otherwise with the standard Lebesgue measure.
Eventually, we use the symbol $H^s(\Omega; \mathbb{C}^4)$ for Sobolev spaces of order $s \geq 0$ 
and $H^1_0(\Omega; \mathbb{C}^4)$ for the closure of $C_c^\infty(\Omega; \mathbb{C}^4)$ with respect
to the $H^1$-norm. For more details on Sobolev and other function spaces, see, e.g.,~\cite{M00}.

The Laplace-Beltrami operator on $\Sigma$ acting on $\mathbb{C}^4$-vector valued functions 
will be denoted by $-\Delta_\Sigma$. The operator 
$(I_4 - \Delta_\Sigma)^s: H^{2 s}(\Sigma; \mathbb{C}^4) \rightarrow L^2(\Sigma; \mathbb{C}^4)$ is bijective and continuous
for any $s \in [-1, 1]$. Finally, we are going to use the following expression for the duality product for the pair 
$H^{1/2}(\Sigma; \mathbb{C}^4)$ and its dual space $H^{-1/2}(\Sigma; \mathbb{C}^4)$:
\begin{equation*}
  (\varphi, \psi)_{1/2 \times -1/2} 
  := \big( (I_4 - \Delta_\Sigma)^{1/4} \varphi, (I_4 - \Delta_\Sigma)^{-1/4} \psi \big)_\Sigma
\end{equation*}
for $\varphi \in H^{1/2}(\Sigma; \mathbb{C}^4)$ and $\psi \in H^{-1/2}(\Sigma; \mathbb{C}^4)$.

\medskip
\noindent {\bf Acknowledgments.} 
The authors wish to thank the referee for helpful comments and remarks
that led to an improvement of the manuscript. Furthermore, 
the authors thank T.~Ourmi{\`e}res-Bonafos and L.~Vega for fruitful discussions.
J. Behrndt gratefully
acknowledges financial support
by the Austrian Science Fund (FWF): Project P~25162-N26. 


\section{Quasi and ordinary boundary triples} \label{section_boundary_triples}

In this section we give a short introduction to ordinary boundary triples, quasi boundary triples, and some related techniques 
in extension and spectral theory of symmetric and self-adjoint operators in Hilbert spaces. 
We formulate the results in a way such that they can be applied directly in the main part of the paper
in the analysis of Dirac operators with singular interactions. In order to get a detailed overview
of the concept of ordinary and quasi boundary triples and applications to partial 
differential operators we refer the reader to
\cite{BL07,BL12,BGP08,DM91,DM95,GG91}.

Throughout this section $\mathcal{H}$ is always a complex Hilbert space with inner product~$(\cdot, \cdot)_\mathcal{H}$
and~$S$ denotes a densely defined, closed and symmetric operator with adjoint~$S^*$.
We start with the definition of quasi and ordinary boundary triples.

\begin{definition}\label{qbtdef}
  Assume that $T$ is a linear operator in $\mathcal{H}$ such that $\overline{T} = S^*$. A triple 
  $\{ \mathcal{G}, \Gamma_0, \Gamma_1 \}$ consisting of
  a Hilbert space $\mathcal{G}$
  and linear mappings $\Gamma_0, \Gamma_1: \dom T \rightarrow \mathcal{G}$
  is called a {\rm quasi boundary triple} for $S^*$ if the following conditions hold:
\begin{itemize}
 \item [{\rm (i)}] For all $f,g\in\dom T$ the abstract Green's identity
 \begin{equation*}
  (Tf,g)_{\mathcal H}-(f,Tg)_{\mathcal H}=(\Gamma_1 f,\Gamma_0 g)_{\mathcal G}-(\Gamma_0 f,\Gamma_1 g)_{\mathcal G}
 \end{equation*}
 is satisfied.
 \item [{\rm (ii)}] $\Gamma=(\Gamma_0,\Gamma_1)^\top:\dom T\rightarrow \mathcal G\times\mathcal G$ has dense range.
 \item [{\rm (iii)}] $A_0:=T\upharpoonright\ker\Gamma_0$ is a self-adjoint operator in $\mathcal H$.
\end{itemize}
If {\rm (i)} and {\rm (iii)} hold, and the mapping $\Gamma=(\Gamma_0,\Gamma_1)^\top:\dom T\rightarrow \mathcal G\times\mathcal G$
is surjective then  
  $\{ \mathcal{G}, \Gamma_0, \Gamma_1 \}$ is called {\rm ordinary boundary triple}.
\end{definition}

We point out that the above (non-standard) definition of ordinary boundary triples is equivalent to the 
usual one in, e.g., \cite{BGP08,DM91,GG91},
see~\cite[Corollary~3.2]{BL07}. 
In particular, if $\{ \mathcal{G}, \Gamma_0, \Gamma_1 \}$ is an ordinary boundary triple,
then $T$ coincides with $S^*$.
Note that a quasi boundary triple or ordinary boundary triple for $S^*$ exists if and only if 
the defect numbers $\dim\ker(S^*\pm i)$ coincide, i.e. if and only if $S$ admits self-adjoint extensions in $\mathcal H$,
and that the operator $T$ in Definition~\ref{qbtdef} is in general 
not unique.

Let $T\subset \overline T=S^*$ and let $\{\mathcal G,\Gamma_0,\Gamma_1\}$ 
be a quasi boundary triple for $S^*$.
Then
\begin{equation*}
 S=T\upharpoonright\bigl(\ker\Gamma_0\cap\ker\Gamma_1\bigr)
\end{equation*}
and the mapping $\Gamma=(\Gamma_0,\Gamma_1)^\top:\dom T\rightarrow\mathcal G\times\mathcal G$ is closable;
cf.~\cite{BL07}. 
Next, we are going to introduce the $\gamma$-field and the Weyl function associated to the quasi 
boundary triple $\{ \mathcal{G}, \Gamma_0, \Gamma_1 \}$;
as we will see one can describe spectral properties of self-adjoint extensions of $S$ with the 
help of these operators.
In the following let $A_0 = T \upharpoonright \ker \Gamma_0$. Then the direct sum decomposition
\begin{equation} \label{decomposition}
  \dom T = \dom A_0 \dot{+} \ker(T - \lambda)=\ker\Gamma_0\dot{+} \ker(T - \lambda),\qquad\lambda\in\rho(A_0),
\end{equation}
holds.
The definition of the $\gamma$-field and Weyl function for quasi boundary triples
is in accordance with the one for ordinary boundary triples in \cite{DM91}.

\begin{definition}\label{gammdef}
 Let $T$ be a linear operator in $\mathcal{H}$ such that $\overline{T} = S^*$ and let 
 $\{ \mathcal{G}, \Gamma_0, \Gamma_1 \}$ be a quasi boundary triple for $S^*$.
 Then the corresponding $\gamma$-field $\gamma$ and Weyl function $M$ are defined by
 \begin{equation*}
  \rho(A_0) \ni \lambda\mapsto\gamma(\lambda)=\bigl(\Gamma_0\upharpoonright\ker(T-\lambda)\bigr)^{-1}
 \end{equation*}
and
 \begin{equation*}
  \rho(A_0) \ni \lambda\mapsto M(\lambda)=\Gamma_1 \bigl(\Gamma_0\upharpoonright\ker(T-\lambda)\bigr)^{-1},
 \end{equation*}
respectively.
\end{definition}

Because of \eqref{decomposition} the $\gamma$-field is well-defined and one has 
$\ran \gamma(\lambda) = \ker(T-\lambda)$ for any $\lambda \in \rho(A_0)$.
Note that $\dom \gamma(\lambda) = \ran \Gamma_0$ is dense in $\mathcal{G}$ 
by Definition~\ref{qbtdef}.
Making use of the abstract Green's formula (Definition~\ref{qbtdef}~(i))
one can show that
\begin{equation}\label{equation_gamma_star}
 \gamma(\lambda)^*=\Gamma_1 (A_0-\overline\lambda)^{-1}, \qquad \lambda \in \rho(A_0);
\end{equation}
this is a bounded and everywhere defined operator from $\mathcal H$ to $\mathcal G$. Thus
$\gamma(\lambda)$ is a (in general not everywhere defined) bounded operator; cf. 
\cite[Proposition~2.6]{BL07} or \cite[Proposition~6.13]{BL12}. In the special case that $\{ \mathcal{G}, \Gamma_0, \Gamma_1 \}$
is an ordinary boundary triple $\gamma(\lambda)$ is automatically bounded and everywhere defined.
Next, one has for all $\lambda, \mu \in \rho(A_0)$ and all $\varphi \in \ran \Gamma_0$ 
\begin{equation}\label{equation_diff_g}
 \gamma(\lambda)\varphi =\bigl(I+(\lambda-\mu)(A_0-\lambda)^{-1}\bigr)\gamma(\mu)\varphi,
\end{equation}
see \cite[Proposition~2.6]{BL07}.
In particular, the mapping $\lambda\mapsto\gamma(\lambda)\varphi$ 
is holomorphic on $\rho(A_0)$ for any fixed $\varphi \in \ran \Gamma_0$.

Next, we state some useful properties of the Weyl function $M$ corresponding to the quasi boundary triple 
$\{\mathcal G,\Gamma_0,\Gamma_1\}$; the proofs of these statements can be found in~\cite[Proposition~2.6]{BL07}.
The definition of $M$ implies  that
\begin{equation*}
 M(\lambda)\Gamma_0 f_\lambda=\Gamma_1 f_\lambda,\qquad f_\lambda\in\ker(T-\lambda), ~\lambda \in \rho(A_0).
\end{equation*}
In particular, for any $\lambda \in \rho(A_0)$ the linear operator $M(\lambda)$ is densely defined in
$\mathcal G$ with $\dom M(\lambda)=\ran\Gamma_0$ and
$\ran M(\lambda)\subset \ran\Gamma_1$.
For $\lambda,\mu\in\rho(A_0)$ and $\varphi\in\ran\Gamma_0$ one has
\begin{equation}\label{equation_diff_m}
 M(\lambda)\varphi-M(\mu)^*\varphi=(\lambda-\overline\mu)\gamma(\mu)^*\gamma(\lambda)\varphi.
\end{equation}
Therefore, we see that $M(\lambda)\subset M(\overline\lambda)^*$ for any $\lambda\in\rho(A_0)$ and hence
$M(\lambda)$ is a closable, but in general unbounded linear operator in 
$\mathcal G$. In the special case that $\{ \mathcal{G}, \Gamma_0, \Gamma_1 \}$
is an ordinary boundary triple 
$M(\lambda)$ is bounded and everywhere defined.
Equation~\eqref{equation_diff_m} also yields that
for any $\varphi\in \ran \Gamma_0$ the $\mathcal G$-valued function $\lambda\mapsto M(\lambda)\varphi$ 
is analytic on~$\rho(A_0)$.

In the main part of the paper we are going to use ordinary boundary triples, quasi boundary triples, 
and their Weyl functions to define and study self-adjoint extensions of the underlying symmetry $S$.
Let $T$ be a linear operator in $\mathcal{H}$ such that $\overline{T} = S^*$,
let $\{ \mathcal{G}, \Gamma_0, \Gamma_1 \}$ be a quasi boundary triple for $S^*$ and
let $\vartheta$ be a linear operator in $\mathcal{G}$. Then, we define the extension $A_\vartheta$ of $S$ by
\begin{equation}\label{equation_def_extension}
 A_\vartheta=T\upharpoonright \ker(\Gamma_1 - \vartheta \Gamma_0),
\end{equation}
i.e. $f\in\dom T$ belongs to $\dom A_\vartheta$ if and only if $f$ satisfies 
$\Gamma_1 f = \vartheta \Gamma_0 f$. 
If $\vartheta$ is a symmetric operator in $\mathcal{G}$ then Green's identity implies 
\begin{equation}\label{abab}
 (A_\vartheta f,g)_{\mathcal H}-(f,A_\vartheta g)_{\mathcal H}=
 (\vartheta \Gamma_0 f, \Gamma_0 g)_{\mathcal G} - (\Gamma_0 f, \vartheta \Gamma_0 g)_{\mathcal G}=0
\end{equation}
for all $f, g \in \dom A_\vartheta$ and hence the extension $A_\vartheta$ is symmetric in $\mathcal H$.

In the following theorem we state an abstract version of the Birman-Schwinger principle and a 
Krein type resolvent formula for canonical extensions $A_\vartheta$; 
for the proof of this result, see~\cite[Theorem~2.8]{BL07} or \cite[Theorem~6.16]{BL12}.

\begin{thm} \label{theorem_Krein_abstract}
Let $T$ be a linear operator in $\mathcal H$ such that $\overline{T} = S^*$, let $\{\mathcal G,\Gamma_0,\Gamma_1\}$
be a quasi boundary triple for $S^*$ with $A_0=T\upharpoonright\ker\Gamma_0$, and denote the corresponding $\gamma$-field and Weyl function 
by $\gamma$ and $M$, respectively.
Let $A_\vartheta$ be the canonical extension of $S$ associated to an operator $\vartheta$ in $\mathcal{G}$
as in \eqref{equation_def_extension}.
Then the following assertions hold for all $\lambda\in\rho(A_0)$:
\begin{itemize}
 \item [\rm (i)] $\lambda \in \sigma_{\mathrm{p}}(A_\vartheta)$ if and only if 
 $0 \in \sigma_{\mathrm{p}}(\vartheta - M(\lambda))$. Moreover, it holds that
 \begin{equation*}
   \ker(A_\vartheta-\lambda)=\bigl\{\gamma(\lambda)\varphi:\varphi\in\ker(\vartheta - M(\lambda))\bigr\}.
 \end{equation*}
 \item [{\rm (ii)}] If $\lambda \notin \sigma_{\mathrm{p}}(A_\vartheta)$ then 
 $g\in\ran(A_\vartheta - \lambda)$ if and only if $\gamma(\overline\lambda)^*g\in\ran(\vartheta - M(\lambda))$.
 \item [{\rm (iii)}] If $\lambda \notin \sigma_{\mathrm{p}}(A_\vartheta)$ then 
 \begin{equation*}
  (A_\vartheta - \lambda)^{-1}g=(A_0-\lambda)^{-1}g
      + \gamma(\lambda)\bigl(\vartheta - M(\lambda)\bigr)^{-1}\gamma(\overline\lambda)^*g
 \end{equation*}
 holds for all $g\in\ran(A_\vartheta-\lambda)$.
\end{itemize}
\end{thm}

Assertion (ii) of the previous theorem shows  how the self-adjointness of an extension $A_\vartheta$
can be proven. Namely, if $\vartheta$ is symmetric in $\mathcal G$ then $A_\vartheta$ is symmetric in $\mathcal H$ by \eqref{abab}, 
and hence $A_\vartheta$ is self-adjoint if, in addition, $\ran (A_\vartheta \mp i) = \mathcal{H}$. According to Theorem~\ref{theorem_Krein_abstract}~(ii)
the latter is equivalent to
$\ran \gamma(\mp i)^* \subset \ran(\vartheta - M(\pm i))$.

In the special case that $\{ \mathcal{G}, \Gamma_0, \Gamma_1 \}$ is an ordinary boundary triple the situation is simpler
as the next well-known proposition states. We note that the converse in Proposition~\ref{proposition_self_adjoint_abstract}
holds if $\vartheta$ in \eqref{equation_def_extension} is allowed to be a linear relation (multivalued operator).

\begin{prop} \label{proposition_self_adjoint_abstract}
Let $S$ be a densely defined closed symmetric operator in $\mathcal H$ and assume that $\{\mathcal G,\Gamma_0,\Gamma_1\}$
is an ordinary boundary triple for $S^*$. Let $\vartheta$ be an operator in $\mathcal{G}$ and let
$A_\vartheta$ be defined by \eqref{equation_def_extension}. If $\vartheta$ is (essentially) self-adjoint in $\mathcal{G}$
then $A_\vartheta$ is (essentially) self-adjoint in $\mathcal H$.
\end{prop}

In what follows we describe a particular transformation procedure of quasi boundary triples to ordinary boundary triples from \cite{BM14}
which will be useful later in  this paper.
Let $T$ be a linear operator such that $\overline{T} = S^*$ and let $\{ \mathcal{G}, \Gamma_0, \Gamma_1 \}$
be a quasi boundary triple for $S^*$.
Define the spaces
\begin{equation} \label{def_G_0_G_1}
  \mathscr{G}_0 := \ran(\Gamma_0 \upharpoonright \ker \Gamma_1) \quad \text{and} \quad
  \mathscr{G}_1 := \ran(\Gamma_1 \upharpoonright \ker \Gamma_0).
\end{equation}
We will often assume that $\mathscr{G}_1$ is dense in $\mathcal{G}$.
In this case, we denote by $\mathscr{G}_1'$ the dual space of $\mathscr{G}_1$ with respect to 
any norm $\| \cdot \|_{\mathscr{G}_1}$ such that $\big(\mathscr{G}_1, \| \cdot \|_{\mathscr{G}_1}\big)$ is 
a reflexive Banach space continuously embedded into $\mathcal{G}$; such a norm exists, see \cite[Proposition~2.9]{BM14},
and all norms with this property are equivalent, cf. \cite[Proposition~2.10]{BM14}.
Analogous statements hold if $\mathscr{G}_0$ is dense in $\mathcal{G}$ and $T \upharpoonright \ker \Gamma_1$
is self-adjoint, and hence we can employ a similar notation in this case as well.

First, it turns out that the boundary mapping $\Gamma_0$ or $\Gamma_1$ can be extended to $\dom S^*$,
if the set $\mathscr{G}_1$ or $\mathscr{G}_0$, respectively, is dense in $\mathcal{G}$;
cf. \cite[Proposition~2.10 and Corollary~2.11]{BM14}. 
In the following we write $\| \cdot \|_{S^*}$ for the graph norm induced by $S^*$.

\begin{prop} \label{proposition_extension_boundary_mappings_abstract}
  Let $T$ be a linear operator such that $\overline{T} = S^*$, let $\{ \mathcal{G}, \Gamma_0, \Gamma_1 \}$
  be a quasi boundary triple for $S^*$, and let $\mathscr G_0,\mathscr{G}_1$ be as in \eqref{def_G_0_G_1}. Then the following assertions are true:
  \begin{itemize}
    \item[\rm (i)] If $\mathscr{G}_1$ is dense in 
    $\mathcal{G}$ then $\Gamma_0$ admits a unique, surjective and continuous extension
    \begin{equation*}
      \widetilde{\Gamma}_0: \big(\dom S^*, \| \cdot \|_{S^*}\big) \rightarrow \mathscr{G}_1'.
    \end{equation*}
    
    \item[\rm (ii)] If $\mathscr{G}_0$ is dense in 
    $\mathcal{G}$ and $A_\infty := T \upharpoonright \ker \Gamma_1$ is self-adjoint 
    then $\Gamma_1$ admits a unique, surjective and continuous extension
    \begin{equation*}
      \widetilde{\Gamma}_1: \big(\dom S^*, \| \cdot \|_{S^*}\big) \rightarrow \mathscr{G}_0'.
    \end{equation*}
  \end{itemize}
\end{prop}

Under the assumptions of the previous proposition also the $\gamma$-field and the Weyl function 
associated to the quasi boundary triple $\{ \mathcal{G}, \Gamma_0, \Gamma_1\}$ can be extended;
cf.~\cite[Definition~2.14]{BM14} and the corresponding discussion.

\begin{prop} \label{proposition_extension_gamma_Weyl_abstract}
  Let $T$ be a linear operator such that $\overline{T} = S^*$, let $\{ \mathcal{G}, \Gamma_0, \Gamma_1 \}$
  be a quasi boundary triple for $S^*$, and let $A_0=T\upharpoonright\ker\Gamma_0$. Denote the corresponding $\gamma$-field and Weyl function 
by $\gamma$ and $M$, respectively.
  Assume that $\mathscr{G}_0$ and $\mathscr{G}_1$ defined by 
  \eqref{def_G_0_G_1} are dense in $\mathcal{G}$ and that $A_\infty := T \upharpoonright \ker \Gamma_1$
  is self-adjoint.  Then the following assertions hold for all $\lambda\in\rho(A_0)$:
  \begin{itemize}
    \item[\rm (i)] The values of the $\gamma$-field admit continuous extensions
    \begin{equation*}
      \widetilde{\gamma}(\lambda) = \big( \widetilde{\Gamma}_0 \upharpoonright \ker(S^* - \lambda) \big)^{-1}: 
      \mathscr{G}_1' \rightarrow \mathcal{H}.
    \end{equation*}
    
    \item[\rm (ii)] The values of the Weyl function $M(\lambda)$ admit continuous extensions
    \begin{equation*}
      \widetilde{M}(\lambda) = \widetilde{\Gamma}_1 \widetilde{\gamma}(\lambda): 
      \mathscr{G}_1' \rightarrow \mathscr{G}_0'.
    \end{equation*}
  \end{itemize}
\end{prop}

Making use of the extended boundary mapping $\widetilde{\Gamma}_0$ one can transform the originally
given quasi boundary triple $\{ \mathcal{G}, \Gamma_0, \Gamma_1 \}$ to an ordinary boundary triple,
see~\cite[Theorem~2.12]{BM14}.
In order to introduce this ordinary boundary triple fix some isomorphisms
$\iota_+: \mathscr{G}_1 \rightarrow \mathcal{G}$ and $\iota_-: \mathscr{G}_1' \rightarrow \mathcal{G}$
which satisfy
\begin{equation*}
  (\iota_- x', \iota_+ x)_{\mathcal{G}} = (x', x)_{\mathscr{G}_1' \times \mathscr{G}_1}
\end{equation*}
for all $x \in \mathscr{G}_1$ and $x' \in \mathscr{G}_1'$, where 
$(\cdot, \cdot)_{\mathscr{G}_1' \times \mathscr{G}_1}$ denotes the duality product of the pair
$\mathscr{G}_1'$ and $\mathscr{G}_1$.

\begin{thm} \label{theorem_extended_boundary_triple_abstract}
  Let $T$ be a linear operator such that $\overline{T} = S^*$ and let $\{ \mathcal{G}, \Gamma_0, \Gamma_1 \}$
  be a quasi boundary triple for $S^*$ with $A_0 = T \upharpoonright \ker \Gamma_0$. Assume that $\mathscr{G}_1$ defined by 
  \eqref{def_G_0_G_1} is dense in $\mathcal{G}$ and that there exists $\mu \in \rho(A_0) \cap \mathbb{R}$.
  Define the mappings $\Upsilon_0, \Upsilon_1: \dom S^* \rightarrow \mathcal{G}$ by
  \begin{equation*}
    \Upsilon_0 f = \iota_- \widetilde{\Gamma}_0 f, \quad \Upsilon_1 f = \iota_+ \Gamma_1 f_0, \quad
    f = f_0 + f_\mu \in \dom A_0 \dot{+} \ker(S^* - \mu),
  \end{equation*}
  where $\widetilde{\Gamma}_0$ is the extension of the boundary
  mapping $\Gamma_0$ from Proposition~\ref{proposition_extension_boundary_mappings_abstract}~{\rm(i)}.
  Then $\{\mathcal{G}, \Upsilon_0, \Upsilon_1 \}$ is an ordinary boundary triple for $S^*$ such that the self-adjoint operators $T\upharpoonright\ker\Gamma_0$ and
  $S^*\upharpoonright\ker\Upsilon_0$ coincide.
\end{thm}

We remark that the $\gamma$-field $\beta$ and the Weyl function $\mathcal{M}$ associated to the boundary triple
$\{ \mathcal{G}, \Upsilon_0, \Upsilon_1\}$ are given by 
\begin{equation} \label{gamma_Weyl_transformed_triple}
  \beta(\lambda) = \widetilde{\gamma}(\lambda) \iota_-^{-1}\quad \text{and} \quad
  \mathcal{M}(\lambda) = \iota_+ \big( \widetilde{M}(\lambda) - \widetilde{M}(\mu) \big) \iota_-^{-1}
\end{equation}
for $\lambda \in \rho(A_0)$ and $\mu \in \mathbb{R} \cap \rho(A_0)$ chosen as in 
Theorem~\ref{theorem_extended_boundary_triple_abstract}; cf.~\cite[eq.~(2.17)]{BM14}.

Finally, let $\vartheta$ be a linear operator in $\mathcal{G}$ and let $A_\vartheta$ be the canonical extension of~$S$ 
defined via \eqref{equation_def_extension} and the quasi boundary triple $\{ \mathcal{G}, \Gamma_0, \Gamma_1 \}$.
Consider the linear operator 
\begin{equation} \label{transformed_parameter} 
\begin{split}
  \Theta(\vartheta) &= \iota_+ \big(\vartheta - M(\mu) \big) \iota_-^{-1},\\
  \dom\Theta(\vartheta)&=\bigl\{\varphi\in\mathcal G:\iota_-^{-1}\varphi\in\dom(\vartheta-M(\mu))\,\,\text{and}\,\,(\vartheta - M(\mu)) \iota_-^{-1}\varphi\in\mathscr G_1\bigr\},
\end{split}
  \end{equation}
in $\mathcal{G}$. If $\{ \mathcal{G}, \Upsilon_0, \Upsilon_1\}$ is the ordinary 
boundary triple in Theorem~\ref{theorem_extended_boundary_triple_abstract} then one verifies 
\begin{equation*}
  \ker (\Gamma_1 - \vartheta \Gamma_0) = \ker \bigl(\Upsilon_1 - \Theta(\vartheta) \Upsilon_0\bigr);
\end{equation*} 
cf. \cite[Corollary~3.5]{BM14}.
Together with Proposition~\ref{proposition_self_adjoint_abstract} the next corollary
follows immediately; again a converse statement is true if $\vartheta$ and $\Theta(\vartheta)$ 
are allowed to be linear relations.

\begin{cor} \label{corollary_self_adjoint_transformed_triple}
  Let $\vartheta$, $\Theta(\vartheta)$ and $A_\vartheta$ be as above and assume that the assumptions in 
  Theorem~\ref{theorem_extended_boundary_triple_abstract} are satisfied. 
  If $\Theta(\vartheta)$ is (essentially) self-adjoint in $\mathcal{G}$
  then the operator 
  $A_\vartheta$ is (essentially) self-adjoint
  in $\mathcal{H}$.
\end{cor}

In this context we also note that
for some self-adjoint operator $\Theta$ acting in $\mathcal{G}$ and its corresponding extension
$A_\Theta = S^*\upharpoonright \ker(\Upsilon_1 - \Theta \Upsilon_0)$ one has
\begin{align} 
  \label{Birman_Schwinger1} 
  \lambda \in \sigma(A_\Theta) \cap \rho(A_0) \quad &\text{if and only if}\quad
  0 \in \sigma(\Theta - \mathcal{M}(\lambda)),\\
  \label{Birman_Schwinger_p} 
  \lambda \in \sigma_{\rm p}(A_\Theta) \cap \rho(A_0) \quad &\text{if and only if}\quad
  0 \in \sigma_{\rm p}(\Theta - \mathcal{M}(\lambda)),
\end{align}
and
\begin{equation} \label{Birman_Schwinger_disc}
  \lambda \in \sigma_{\rm disc}(A_\Theta) \cap \rho(A_0)\quad \text{if and only if}\quad
  0 \in \sigma_{\rm disc}(\Theta - \mathcal{M}(\lambda));
\end{equation} 
cf.~\cite{DM91,DM95} and~\cite[Theorem~1.29 and Theorem~3.3]{BGP08}. Moreover, for $\lambda \in \rho(A_0)\cap\rho(A_\Theta)$
we have
\begin{equation} \label{Krein_transformed}
  (A_\Theta - \lambda)^{-1}=(A_0-\lambda)^{-1}
      + \beta(\lambda)\bigl(\Theta - \mathcal{M}(\lambda)\bigr)^{-1}\beta(\overline\lambda)^*;
\end{equation}
see~\cite{DM91,DM95} and \cite[Section~3]{BM14} for more details.

\section{The free and the maximal Dirac operator} 
\label{section_free_Dirac}

In this section we first recall the definition and some standard properties of the free Dirac operator,
which will be of importance in our further considerations. Then we introduce and discuss the maximal Dirac 
operator in $\mathbb{R}^3 \setminus \Sigma$, where $\Sigma$ is the boundary of a bounded $C^2$-domain.

Let us choose units such that the speed of light and the Planck constant $\hbar$ are both equal to one.
Then, the free Dirac operator is given by
\begin{equation} \label{def_free_Dirac}
  A_0 f := -i \sum_{j=1}^3 \alpha_j \partial_j f + m \beta f=-i\alpha\cdot\nabla f+m\beta f, 
  \quad \dom A_0 = H^1(\mathbb{R}^3; \mathbb{C}^4),
\end{equation}
where the Dirac matrices $\alpha_1, \alpha_2, \alpha_3$ and $\beta$ are defined by \eqref{def_Dirac_matrices} 
and we require $m>0$.
If $-\Delta$ denotes the self-adjoint Laplace operator in $L^2(\mathbb{R}^3; \mathbb{C})$ 
defined on $H^2(\mathbb{R}^3; \mathbb{C})$ then 
\begin{equation} \label{equation_A_0_square}
  A_0^2  = (-\Delta + m^2) I_4, \qquad \dom A_0^2 = H^2(\mathbb{R}^3; \mathbb{C}^4);
\end{equation}
cf. \cite[Korollar~20.2]{W03} for $m=1$. 
In the above formula the symbol $(-\Delta + m^2) I_4$ is understood as a $4 \times 4$ diagonal block operator, 
where each diagonal entry acts as $-\Delta + m^2$.
Next, it is well-known that $A_0$ is self-adjoint in $L^2(\mathbb{R}^3; \mathbb{C}^4)$ and that 
the spectrum of $A_0$ is
\begin{equation} \label{equation_spectrum_A_0}
  \sigma(A_0) = (-\infty, -m] \cup [m, \infty),
\end{equation}
see \cite{T92} or \cite[Chapter~20]{W03}. 
Furthermore, for $\lambda \notin \sigma(A_0)$ the resolvent of $A_0$ 
is given by
\begin{equation} \label{equation_resolvent_A_0}
(A_0 - \lambda)^{-1} f(x) = \int_{\mathbb{R}^3} G_\lambda(x - y) f(y) \mathrm{d} y, 
\quad x \in \mathbb{R}^3,~f \in L^2(\mathbb{R}^3; \mathbb{C}^4),
\end{equation}
where the integral kernel $G_\lambda$ is a $\mathbb C^{4\times 4}$-valued function of the form
\begin{equation} \label{def_G_lambda}
	G_\lambda(x) = \left( \lambda I_4 + m \beta 
	+ \left( 1 - i \sqrt{\lambda^2 - m^2} |x| \right) \frac{i(\alpha \cdot x )}{|x|^2} \right)
	\frac{e^{i \sqrt{\lambda^2 - m^2} |x|}}{4 \pi |x|};
\end{equation}
cf.~\cite[Section~1.E]{T92} or \cite[Lemma~2.1]{AMV15}. In the above formula the square root is defined such that 
$\mathrm{Im}\, \sqrt{\lambda^2 - m^2} > 0$ for $\lambda\not\in\sigma(A_0)$.

Let $\Sigma$ be the boundary of the bounded $C^2$-domain $\Omega_+$ and let 
$\Omega_- := \mathbb{R}^3 \setminus \overline{\Omega_+}$.
We will make  use of the decomposition
$L^2(\mathbb{R}^3; \mathbb{C}^4) = L^2(\Omega_+; \mathbb{C}^4) \oplus L^2(\Omega_-; \mathbb{C}^4)$ and 
split functions $f \in L^2(\mathbb{R}^3; \mathbb{C}^4)$ in the form 
$f = f_+ \oplus f_-$, where $f_\pm := f \upharpoonright \Omega_\pm \in L^2(\Omega_\pm; \mathbb{C}^4)$.
Furthermore, we define the subspaces $\mathcal{D}_\pm$ of $L^2(\Omega_\pm; \mathbb{C}^4)$ by
\begin{equation*}
  \mathcal{D}_\pm := \left \{ f_\pm \in L^2(\Omega_\pm; \mathbb{C}^4):
                        (-i \alpha \cdot \nabla + m \beta) f_\pm \in L^2(\Omega_\pm; \mathbb{C}^4) \right\},
\end{equation*}
where all derivatives are understood in the distributional sense, and we endow $\mathcal D_\pm$ with the natural norms
\begin{equation}\label{dpm}
 \| f_\pm \|_{\mathcal{D}_\pm}^2:=\Vert f_\pm\Vert_{\Omega_\pm}^2
 +\bigl\Vert (-i \alpha \cdot \nabla + m \beta) f_\pm\bigr\Vert_{\Omega_\pm}^2,\quad
 f_\pm\in\mathcal D_\pm.
\end{equation}

Now, we define the maximal Dirac operator $T_{\rm max}$ in $L^2(\mathbb{R}^3; \mathbb{C}^4)$ by
\begin{equation} \label{def_T_m}
  \begin{split}
    T_{\rm max} f &:= (-i \alpha \cdot \nabla + m \beta) f_+ \oplus (-i \alpha \cdot \nabla + m \beta) f_-, \\
    \dom T_{\rm max} &:= \mathcal{D}_+ \oplus \mathcal{D}_-.
  \end{split}
\end{equation}
The operator $T_{\rm{max}}$ turns out to be the adjoint of the symmetric restriction of $A_0$ on functions vanishing on $\Sigma$. 

\begin{prop} \label{corollary_S_star}
  Define the linear operator $S$ by
  \begin{equation} \label{sss}
    S := A_0 \upharpoonright H^1_0(\mathbb{R}^3 \setminus \Sigma; \mathbb{C}^4)
  \end{equation}
  and let $T_{\rm max}$ be as above. Then $S$ is a densely defined, closed, symmetric operator  such that $S^* = T_{\rm max}$ holds.
\end{prop}
\begin{proof}
  Since $C_c^\infty(\mathbb{R}^3 \setminus \Sigma; \mathbb{C}^4)\subset\dom S$ and $S\subset A_0$ it is clear that $S$ is densely defined and symmetric.
  Moreover $S$ is closed since $H^1_0(\mathbb{R}^3 \setminus \Sigma; \mathbb{C}^4)$ is a closed subspace of 
  $H^1(\mathbb{R}^3; \mathbb{C}^4)$
  and the graph norm of $A_0$ is equivalent 
  to the $H^1(\mathbb{R}^3; \mathbb{C}^4)$-norm, see, e.g., \cite[Satz~20.1]{W03}.

  Next we show $S^* \subset T_{\rm{max}}$. For that, let $f \in \dom S^*$ and 
  $g_+ \in C_c^\infty(\Omega_+; \mathbb{C}^4)$. Then $g := g_+ \oplus 0 \in \dom S$ and 
  \begin{equation*} 
    \big((S^* f)_+, g_+\big)_{\Omega_+} = (S^* f, g)_{\mathbb{R}^3} 
    = (f, S g)_{\mathbb{R}^3} = \big( f_+, (-i \alpha \cdot \nabla + m \beta) g_+ \big)_{\Omega_+}.
  \end{equation*}
  Since this holds for any $g_+ \in C_c^\infty(\Omega_+; \mathbb{C}^4)$, the distribution
  $(-i \alpha \cdot \nabla + m \beta) f_+$ exists in $L^2(\Omega_+; \mathbb{C}^4)$ and is equal to $(S^* f)_+$.
  Similarly, one verifies $(-i \alpha \cdot \nabla + m \beta) f_- = (S^* f)_-$ in the distributional sense.
  This yields $f \in \mathcal{D}_+ \oplus \mathcal{D}_- = \dom T_{\rm{max}}$ and 
  $T_{\rm{max}} f = S^* f$.
  
  It remains to prove that $T_{\rm{max}} \subset S^*$. Let $f \in \dom T_{\rm{max}}$ 
  and $g \in C_c^\infty(\mathbb{R}^3 \setminus \Sigma; \mathbb{C}^4)$. Then we have
  \begin{equation*}
    \begin{split}
      \big(f_+, (S g)_+\big)_{\Omega_+} &= \big( f_+, (-i \alpha \cdot \nabla + m \beta) g_+ \big)_{\Omega_+} \\
      &= \big( (-i \alpha \cdot \nabla + m \beta) f_+, g_+ \big)_{\Omega_+} = \big( (T_{\rm{max}} f)_+, g_+ \big)_{\Omega_+}
    \end{split}
  \end{equation*}
  and similarly $(f_-, (S g)_-)_{\Omega_-} = \big( (T_{\rm{max}} f)_-, g_- \big)_{\Omega_-}$.
  Summing up these two equations yields
  \begin{equation*}
    (f, S g)_{\mathbb{R}^3} = (T_{\rm{max}} f, g)_{\mathbb{R}^3}.
  \end{equation*}
  A density argument shows that this remains valid for any 
  $g \in H^1_0(\mathbb{R}^3 \setminus \Sigma; \mathbb{C}^4) = \dom S$
  and hence $f \in \dom S^*$ and $S^* f = T_{\rm{max}} f$.
  This completes the proof of this proposition.
\end{proof}

The next lemma implies that smooth functions are dense in $\dom T_{\rm max}$ equipped with the graph norm.
The proof follows the strategy in \cite[Lemma~2.1]{BFSB17_1}; a similar result can also be found 
in~\cite[Proposition~2.12]{OV16}.

\begin{lem} \label{lemma_core_T}
  The space $C^\infty(\overline{\Omega_\pm}; \mathbb{C}^4)$ is dense in 
  $\mathcal{D}_\pm$ with respect to the norm $\| \cdot \|_{\mathcal{D}_\pm}$ in \eqref{dpm}.
\end{lem}
\begin{proof}
  We show this statement for $\mathcal{D}_-$, the assertion for $\mathcal{D}_+$ follows in almost the same way.
  Assume that $f \in \mathcal{D}_-$ satisfies
  \begin{equation} \label{equation_orthogonal_smooth}
    \big( f, g \big)_{\Omega_-} 
        + \big( (-i \alpha \cdot \nabla + m \beta) f, (-i \alpha \cdot \nabla + m \beta) g \big)_{\Omega_-} 
        = 0
  \end{equation}
  for all $g \in C^\infty(\overline{\Omega_-}; \mathbb{C}^4)$. 
  Since this is true, in particular, for any $g \in C_c^\infty(\Omega_-; \mathbb{C}^4)$, it follows that the distribution
  $(-i \alpha \cdot \nabla + m \beta)^2 f$ exists in $L^2(\Omega_-; \mathbb{C}^4)$ and coincides with $-f$.
  
  Next, we claim that 
  \begin{equation}\label{jabitte}
  (-i \alpha \cdot \nabla + m \beta) f \in H^1_0(\Omega_-; \mathbb{C}^4).
  \end{equation}
  To see this, let $A_0$ be the free Dirac operator in \eqref{def_free_Dirac}, 
  let $h \in C_c^\infty(\mathbb{R}^3; \mathbb{C}^4)$ and choose a smooth cutoff function $\chi: \mathbb{R}^3 \rightarrow [0, 1]$
  satisfying $\chi \equiv 1$ in $B(0, 1)$ and $\chi \equiv 0$ in $\mathbb{R}^3 \setminus B(0, 2)$.
  Set $\chi_l := \chi(\cdot/l)$, $l \in \mathbb{N}$. Then $(\chi_l A_0^{-1} h) \!\upharpoonright\! \Omega_-\in C^\infty(\overline{\Omega_-}; \mathbb{C}^4)$ 
  converges to $(A_0^{-1} h)_-$ in $H^1(\Omega_-; \mathbb{C}^4)$ as $l \rightarrow \infty$.
  Making use of \eqref{equation_orthogonal_smooth}, we conclude that
  \begin{equation*}
    \begin{split}
      \big( A_0^{-1} (0 \oplus -f),h \big)_{\mathbb{R}^3} 
          &= -\big( f, (A_0^{-1} h)_- \big)_{\Omega_-} 
          = -\lim_{l \rightarrow \infty} \big( f, (\chi_l A_0^{-1} h)_- \big)_{\Omega_-} \\
      &= \lim_{l \rightarrow \infty} \big( (-i \alpha \cdot \nabla + m \beta) f, 
          (-i \alpha \cdot \nabla + m \beta) (\chi_l A_0^{-1} h)_- \big)_{\Omega_-} \\
      &= \big( (-i \alpha \cdot \nabla + m \beta) f, 
          (-i \alpha \cdot \nabla + m \beta) (A_0^{-1} h)_- \big)_{\Omega_-} \\
      & = \big( 0 \oplus (-i \alpha \cdot \nabla + m \beta) f, h \big)_{\mathbb{R}^3}.
    \end{split}
  \end{equation*}
  Since this is true for any $h \in C_c^\infty(\mathbb{R}^3; \mathbb{C}^4)$ it follows that 
  \begin{equation*}
    0 \oplus (-i \alpha \cdot \nabla + m \beta) f = A_0^{-1} (0 \oplus -f) \in H^1(\mathbb{R}^3; \mathbb{C}^4).
  \end{equation*}
  Moreover, the trace of $0 \oplus (-i \alpha \cdot \nabla + m \beta) f$ at $\Sigma$ is equal to zero. This yields \eqref{jabitte}.
  
  From \eqref{jabitte} it is clear that there exists a sequence $(h_n) \subset C_c^\infty(\Omega_-; \mathbb{C}^4)$ such that 
  $h_n \rightarrow (-i \alpha \cdot \nabla + m \beta) f$ in $H^1(\Omega_-; \mathbb{C}^4)$. 
  Integration by parts yields finally that
  \begin{equation*}
    \begin{split}
      0 &\leq \big( (-i \alpha \cdot \nabla + m \beta) f, 
          (-i \alpha \cdot \nabla + m \beta) f \big)_{\Omega_-} 
          = \lim_{n \rightarrow \infty} \big( h_n, (-i \alpha \cdot \nabla + m \beta) f \big)_{\Omega_-} \\
      &= \lim_{n \rightarrow \infty} \big( (-i \alpha \cdot \nabla + m \beta) h_n, f \big)_{\Omega_-} 
          = \big( (-i \alpha \cdot \nabla + m \beta)^2 f, f \big)_{\Omega_-} \\
      &= \big( -f, f \big)_{\Omega_-} \leq 0.
    \end{split}
  \end{equation*}
  Thus, $f = 0$ and hence $C^\infty(\overline{\Omega_-}; \mathbb{C}^4)$ is dense in 
  $\mathcal{D}_-$.
\end{proof}

\section{Boundary triples for Dirac operators with $\delta$-shell interactions}
\label{section_boundary_triples_Dirac}

In this section we first provide a quasi boundary triple 
which is convenient to study Dirac operators with singular interactions supported on~$\Sigma$.
In a slightly different way this quasi boundary triple was already introduced in \cite{BEHL16_2}. Although only  $C^\infty$-smooth 
surfaces~$\Sigma$ were considered in \cite{BEHL16_2} the relevant results below remain valid for $C^2$-surfaces.
The main purpose is then to  extend and transform this quasi boundary triple to an ordinary boundary 
triple as explained in Section~\ref{section_boundary_triples}; cf. \cite{BM14}.

First, we define the operator 
$T := T_{\rm max} \upharpoonright H^1(\mathbb{R}^3 \setminus \Sigma; \mathbb{C}^4)$, 
that is, 
\begin{equation} \label{def_T}
  \begin{split}
   T f &:= (-i \alpha \cdot \nabla + m \beta) f_+ \oplus (-i \alpha \cdot \nabla + m \beta) f_-,\\
    \dom T &:= H^1(\mathbb{R}^3 \setminus \Sigma; \mathbb{C}^4) = H^1(\Omega_+; \mathbb{C}^4) \oplus H^1(\Omega_-; \mathbb{C}^4),
  \end{split}
\end{equation}
and the linear mappings $\Gamma_0, \Gamma_1: \dom T \rightarrow L^2(\Sigma; \mathbb{C}^4)$ by
\begin{equation} \label{def_Gamma}
  \Gamma_0 f = i \alpha \cdot \nu (f_+|_\Sigma - f_-|_\Sigma) \quad \text{and} \quad
  \Gamma_1 f = \frac{1}{2} (f_+|_\Sigma + f_-|_\Sigma),
  \quad f \in \dom T.
\end{equation}
Since $\Sigma$ is $C^2$-smooth, it follows that the normal vector field $\nu$ is differentiable and 
  hence, $\Gamma_0 f, \Gamma_1 f \in H^{1/2}(\Sigma; \mathbb{C}^4)$ for 
  $f \in \dom T$ by the trace theorem.

It can be deduced from \cite{BEHL16_2} that $\{ L^2(\Sigma; \mathbb{C}^4), \Gamma_0, \Gamma_1 \}$ is a quasi boundary 
triple for $\overline{T}$ (see \cite[Remark~3.3]{BEHL16_2}).
For the convenience of the reader we give a direct and simple proof here.

\begin{thm} \label{theorem_triple}
  Let $S$ be given by \eqref{sss} and let the mappings $T, \Gamma_0$ and $\Gamma_1$ be as in \eqref{def_T} and
  \eqref{def_Gamma}, respectively. 
  Then $\{ L^2(\Sigma; \mathbb{C}^4), \Gamma_0, \Gamma_1 \}$ is a quasi boundary triple
  for $T_{\rm{max}}=S^* = \overline{T}$ with 
  \begin{equation}\label{dr}
   \ran (\Gamma_0,\Gamma_1)^\top=H^{1/2}(\Sigma; \mathbb{C}^4) \times H^{1/2}(\Sigma; \mathbb{C}^4),
  \end{equation}
  and $T \upharpoonright \ker \Gamma_0$
  is the free Dirac operator $A_0$ in \eqref{def_free_Dirac}. 
\end{thm}

\begin{proof}
 Observe first that $\overline{T} = T_{\rm max}$ since $C^\infty(\overline{\Omega_+}; \mathbb{C}^4) \oplus C^\infty(\overline{\Omega_+}; \mathbb{C}^4)$
  is dense in $\dom T_{\rm max}$ with respect to the graph norm by Lemma~\ref{lemma_core_T}. Hence also the space
  $H^1(\Omega_+; \mathbb{C}^4) \oplus H^1(\Omega_+; \mathbb{C}^4)$
  is dense in $\dom T_{\rm max}$ with respect to the graph norm and thus $\overline{T} = T_{\rm max}$.
  Moreover, $S$ is closed and $S^*=T_{\rm max}$ by Proposition~\ref{corollary_S_star}.

  Next it will be shown that Green's identity holds.
  For this let $f = f_+ \oplus f_-$, $g = g_+ \oplus g_- \in \dom T = H^1(\Omega_+; \mathbb{C}^4) \oplus H^1(\Omega_-; \mathbb{C}^4)$.
  Using $(-i \alpha_j)^* = i \alpha_j$, $j \in \{1, 2, 3 \}$, we get by integration by parts
  \begin{equation*}
    \big( (-i \alpha \cdot \nabla + m \beta) f_\pm, g_\pm \big)_{\Omega_\pm}
     - \big( f_\pm, (-i \alpha \cdot \nabla + m \beta) g_\pm \big)_{\Omega_\pm}
     = \pm \big( -i \alpha \cdot \nu f_\pm|_\Sigma, g_\pm|_\Sigma \big)_\Sigma;
  \end{equation*}
  note that the normal vector field $\nu$ always points inside $\Omega_-$, hence there is a different sign 
  on the right hand side. By adding these two formulae for $\Omega_+$ and $\Omega_-$,
  we obtain
  \begin{equation*}
    (T f, g)_{\mathbb{R}^3} - (f, T g)_{\mathbb{R}^3} = (\Gamma_1 f, \Gamma_0 g)_\Sigma - (\Gamma_0 f, \Gamma_1 g)_\Sigma,
  \end{equation*}
  i.e. Green's identity in Definition~\ref{qbtdef}~(i) is valid.
  
  To prove the range property \eqref{dr} consider $\varphi, \psi\in H^{1/2}(\Sigma; \mathbb{C}^4)$. 
  By the trace theorem and $(\alpha \cdot \nu)^2 = I_4$ 
  there exists $g_+ \in H^1(\Omega_+; \mathbb{C}^4)$ and $h \in H^1(\mathbb{R}^3; \mathbb{C}^4)$ such that 
  \begin{equation*}
    i \alpha \cdot \nu g_+|_\Sigma = \varphi\quad\text{and}\quad h|_\Sigma = \psi - \frac{1}{2} g_+|_\Sigma.
  \end{equation*}
  Then $f := (g_+ \oplus 0) + h$ belongs to $\dom T = H^1(\mathbb{R}^3 \setminus \Sigma; \mathbb{C}^4)$
  and satisfies 
  \begin{equation*}
    \Gamma_0 f =  i \alpha \cdot \nu g_+|_\Sigma +  i \alpha \cdot \nu (h_+|_\Sigma - h_-|_\Sigma) = \varphi
    \quad\text{and}\quad \Gamma_1 f = \frac{1}{2} g_+|_\Sigma + h|_\Sigma = \psi.
  \end{equation*}
  This implies \eqref{dr} and hence item~(ii) in Definition~\ref{qbtdef}. 
  Finally, since $\ker \Gamma_0 = H^1(\mathbb{R}^3; \mathbb{C}^4)$ the restriction $T \upharpoonright \ker \Gamma_0$
  coincides with the free Dirac operator $A_0$ which is self-adjoint. Therefore, 
  $\{ L^2(\Sigma; \mathbb{C}^4), \Gamma_0, \Gamma_1 \}$ 
  is a quasi boundary triple for~$S^*$. 
\end{proof}

Next we provide the $\gamma$-field 
and Weyl function associated to the quasi boundary triple 
in Theorem~\ref{theorem_triple}. 

\begin{prop} \label{proposition_gamma_Weyl}
  Let $\{L^2(\Sigma;\mathbb C^4), \Gamma_0, \Gamma_1\}$ be the quasi boundary triple in Theorem~\ref{theorem_triple},
  let $\lambda \in \rho(A_0)= \mathbb{C} \setminus ( (-\infty, -m] \cup [m, \infty) )$
  and let $G_\lambda$ be the integral kernel of the resolvent of the free Dirac operator in \eqref{def_G_lambda}. 
  Then the following statements hold.
  \begin{itemize}
    \item[\rm (i)] The values $\gamma(\lambda): L^2(\Sigma; \mathbb{C}^4) \rightarrow L^2(\mathbb{R}^3; \mathbb{C}^4)$  of the $\gamma$-field 
    are defined on $H^{1/2}(\Sigma; \mathbb{C}^4)$ and given by
    \begin{equation*} 
      \begin{split}
        \gamma(\lambda) \varphi(x) = \int_\Sigma G_\lambda(x - y) \varphi(y) \mathrm{d} \sigma(y),
        \quad x \in \mathbb{R}^3, ~\varphi \in H^{1/2}(\Sigma; \mathbb{C}^4).
      \end{split}
    \end{equation*}
    Each $\gamma(\lambda)$ is a densely defined and bounded operator 
    from $L^2(\Sigma; \mathbb{C}^4)$ to $L^2(\mathbb{R}^3; \mathbb{C}^4)$ and 
    an everywhere defined  bounded operator from $H^{1/2}(\Sigma; \mathbb{C}^4)$ to 
    $H^1(\mathbb{R}^3 \setminus \Sigma; \mathbb{C}^4)$.
    The adjoint  $\gamma(\lambda)^*: L^2(\mathbb{R}^3; \mathbb{C}^4) \rightarrow L^2(\Sigma; \mathbb{C}^4)$ is 
    \begin{equation*}
      \gamma(\lambda)^* f(x) = \int_{\mathbb{R}^3} G_{\overline{\lambda}}(x - y) f(y) \mathrm{d} y,
      \quad x \in \Sigma, ~f \in L^2(\mathbb{R}^3; \mathbb{C}^4).
    \end{equation*}
   
    \item[\rm (ii)] The values $M(\lambda): L^2(\Sigma; \mathbb{C}^4) \rightarrow L^2(\Sigma; \mathbb{C}^4)$ of the Weyl function 
    are defined on $H^{1/2}(\Sigma; \mathbb{C}^4)$ and given by
    \begin{equation*}
      \begin{split}
        &M(\lambda) \varphi(x) := 
        \lim_{\varepsilon \searrow 0} \int_{|x - y| > \varepsilon} G_{\lambda}(x - y) \varphi(y) \mathrm{d} \sigma(y),
        \quad x \in \Sigma, ~\varphi \in H^{1/2}(\Sigma; \mathbb{C}^4).
      \end{split}
    \end{equation*}
    Each $M(\lambda)$ is a densely defined bounded operator in $L^2(\Sigma; \mathbb{C}^4)$ and an everywhere defined  bounded operator
   in $H^{1/2}(\Sigma; \mathbb{C}^4)$.
  \end{itemize}
\end{prop}

\begin{proof}
Since the triple $\{ L^2(\Sigma; \mathbb{C}^4), \Gamma_0, \Gamma_1 \}$ in Theorem~\ref{theorem_triple} is a restriction of the quasi boundary triple 
considered in \cite{BEHL16_2}, the $\gamma$-field, the Weyl function and the adjoint $\gamma$-field are restrictions of the corresponding operators
there; their explicit computation can be found in \cite[Proposition~3.4]{BEHL16_2}.
By definition and from \cite[Proposition~3.4]{BEHL16_2} it is also clear that $\gamma(\lambda)$ and $M(\lambda)$
  are defined on $\ran \Gamma_0 = H^{1/2}(\Sigma; \mathbb{C}^4)$ and are bounded operators in the respective $L^2$-spaces.
  Moreover, by Definition~\ref{gammdef} we have
  $\ran \gamma(\lambda) = \ker(T-\lambda) \subset H^1(\mathbb{R}^3 \setminus \Sigma; \mathbb{C}^4)$ and 
  $\ran M(\lambda) \subset \ran \Gamma_1 \subset H^{1/2}(\Sigma; \mathbb{C}^4)$.
  
  To see that $\gamma(\lambda): H^{1/2}(\Sigma; \mathbb{C}^4) \rightarrow H^1(\mathbb{R}^3 \setminus \Sigma; \mathbb{C}^4)$
  is bounded it suffices to show that $\gamma(\lambda)$ is closed. Assume that $(\varphi_n) \subset H^{1/2}(\Sigma; \mathbb{C}^4)$ is a sequence such that 
  \begin{equation*}
    \varphi_n \rightarrow \varphi \quad \text{in } H^{1/2}(\Sigma; \mathbb{C}^4)\quad \text{ and}\quad
    \gamma(\lambda) \varphi_n \rightarrow f \quad \text{in } H^1(\mathbb{R}^3 \setminus \Sigma; \mathbb{C}^4).
  \end{equation*}
  Clearly, $\varphi \in H^{1/2}(\Sigma; \mathbb{C}^4) =  \dom \gamma(\lambda)$ and $\varphi_n \rightarrow \varphi$ in $L^2(\Sigma; \mathbb{C}^4)$.
  Since $\gamma(\lambda)$ is bounded in the respective $L^2$-spaces, we have 
  $\gamma(\lambda) \varphi_n \rightarrow \gamma(\lambda) \varphi$ in $L^2(\mathbb{R}^3; \mathbb{C}^4)$.
  This implies $\gamma(\lambda) \varphi = f$ and therefore
  $\gamma(\lambda): H^{1/2}(\Sigma; \mathbb{C}^4) \rightarrow H^1(\mathbb{R}^3 \setminus \Sigma; \mathbb{C}^4)$
  is closed and everywhere defined, and hence bounded.
  
  Finally, since $\gamma(\lambda): H^{1/2}(\Sigma; \mathbb{C}^4) \rightarrow H^1(\mathbb{R}^3 \setminus \Sigma; \mathbb{C}^4)$
  is bounded, the continuity of the operator $M(\lambda) = \Gamma_1 \gamma(\lambda)$ in $H^{1/2}(\Sigma; \mathbb{C}^4)$
  follows from the continuity of the trace map. 
\end{proof}

Recall that $M(\lambda)$ is injective for any $\lambda \in \mathbb{C} \setminus \big( (-\infty, -m] \cup [m, \infty) \big)$ 
and that its inverse is given by
\begin{equation} \label{equation_M_inv}
  M(\lambda)^{-1} = -4 \alpha \cdot \nu M(\lambda) \alpha \cdot \nu.
\end{equation}
In particular, $M(\lambda)$ is bijective in $H^{1/2}(\Sigma; \mathbb{C}^4)$.
For $\lambda \in (-m, m)$ equation \eqref{equation_M_inv}
follows from the known identity $-4 (M(\lambda) \alpha \cdot \nu)^2 = I_4$, 
which is stated, e.g., in \cite[Lemma~2.2~(ii)]{AMV15} (note that $M(\lambda) = C_\sigma^\lambda$ 
in the notation of \cite[Lemma~2.2]{AMV15}).
For $\lambda \in \mathbb{C} \setminus \mathbb{R}$ the above 
formula~\eqref{equation_M_inv} follows then by an analytic continuation argument.

Next we extend and transform the 
quasi boundary triple from Theorem~\ref{theorem_triple} to an ordinary boundary triple for $S^*$ 
using Proposition~\ref{proposition_extension_boundary_mappings_abstract} and Theorem~\ref{theorem_extended_boundary_triple_abstract}.
Recall from \eqref{def_G_0_G_1} that $\mathscr{G}_0$ and $\mathscr{G}_1$ are defined by
\begin{equation}\label{ggg}
  \mathscr{G}_0 := \ran (\Gamma_0 \upharpoonright \ker \Gamma_1) \quad\text{and}\quad 
  \mathscr{G}_1 := \ran (\Gamma_1 \upharpoonright \ker \Gamma_0).
\end{equation} 

\begin{lem} \label{lemma_A_infty}
  Let $\{ L^2(\Sigma; \mathbb{C}^4), \Gamma_0, \Gamma_1 \}$ be the quasi boundary triple in Theorem~\ref{theorem_triple}.
  Then the operator $A_\infty := T \upharpoonright \ker \Gamma_1$ is self-adjoint, the spaces $\mathscr{G}_0$ 
  and $\mathscr{G}_1$ in \eqref{ggg} are  
  \begin{equation}\label{ghj}
      \mathscr{G}_0 = \mathscr{G}_1  = H^{1/2}(\Sigma; \mathbb{C}^4),
    \end{equation}  
   and $\Gamma_0$ and $\Gamma_1$ have surjective extensions
  \begin{equation*}
    \widetilde{\Gamma}_0: \dom T_{\rm max} \rightarrow H^{-1/2}(\Sigma; \mathbb{C}^4) \quad \text{and} \quad
    \widetilde{\Gamma}_1: \dom T_{\rm max} \rightarrow H^{-1/2}(\Sigma; \mathbb{C}^4),
  \end{equation*}
  which are continuous with respect to the graph norm of $T_{\rm max}$. 
\end{lem}
\begin{proof}
  First, it will be shown that $A_\infty$ is self-adjoint.
  Because of Green's formula (Definition~\ref{qbtdef}~(i)) we see immediately that $A_\infty$ is symmetric. 
  To prove that 
  $A_\infty$ is self-adjoint it suffices to check  $\ran A_\infty = L^2(\mathbb{R}^3; \mathbb{C}^4)$,
  which by Theorem~\ref{theorem_Krein_abstract}~(ii) is the case if and only if $\ran \gamma(0)^* \subset \ran M(0)$.
  The latter inclusion holds since $\ran \gamma(0)^*=\ran \Gamma_1 A_0^{-1}=H^{1/2}(\Sigma; \mathbb{C}^4)$ by \eqref{equation_gamma_star} 
  and \eqref{equation_M_inv} yields that $M(0)$ is bijective in $H^{1/2}(\Sigma; \mathbb{C}^4)$. 
  Therefore $A_\infty$ is self-adjoint.

  In order to show \eqref{ghj} 
  let $\varphi \in H^{1/2}(\Sigma; \mathbb{C}^4)$ and choose functions $f_\pm \in H^1(\Omega_\pm; \mathbb{C}^4)$
  with $f_\pm|_\Sigma = \mp \frac{1}{2} i \alpha \cdot \nu \varphi$. Then 
  $f = f_+ \oplus f_- \in \ker \Gamma_1$ and $\Gamma_0 f = \varphi$, and hence $\mathscr{G}_0 =  H^{1/2}(\Sigma; \mathbb{C}^4)$. 
  To show the claim on $\mathscr{G}_1$ consider  
  $\varphi \in H^{1/2}(\Sigma; \mathbb{C}^4)$ and choose $f \in H^1(\mathbb{R}^3; \mathbb{C}^4)$
  with $f|_\Sigma = \varphi$.
  Then $f \in \ker \Gamma_0$ and $\Gamma_1 f = f|_\Sigma = \varphi$. Hence \eqref{ghj} is shown. 
  
  The last assertion on the surjective extensions of $\Gamma_0$ and $\Gamma_1$ is now an immediate 
  consequence of Proposition~\ref{proposition_extension_boundary_mappings_abstract}.
  \end{proof}

Next we discuss the extensions of the $\gamma$-field and the Weyl function of the quasi boundary triple $\{ L^2(\Sigma; \mathbb{C}^4), \Gamma_0, \Gamma_1 \}$.

\begin{prop} \label{proposition_extension_gamma_Weyl}
  Let $\{ L^2(\Sigma; \mathbb{C}^4), \Gamma_0, \Gamma_1 \}$ be the quasi boundary triple from Theorem~\ref{theorem_triple}
  with corresponding $\gamma$-field $\gamma$ and Weyl function $M$.
  Then the following assertions hold for all $\lambda\in\rho(A_0)$:
  \begin{itemize}
    \item[\rm (i)] The values of the  $\gamma$-field admit continuous extensions
    \begin{equation*}
      \widetilde{\gamma}(\lambda) = \big( \widetilde{\Gamma}_0 \upharpoonright \ker(T_{\rm max} - \lambda) \big)^{-1}:
      H^{-1/2}(\Sigma; \mathbb{C}^4) \rightarrow L^2(\mathbb{R}^3; \mathbb{C}^4).
    \end{equation*}
    
    \item[\rm (ii)] The values of the Weyl function admit continuous extensions
    \begin{equation*}
      \widetilde{M}(\lambda) = \widetilde{\Gamma}_1\big( \widetilde{\Gamma}_0 \upharpoonright \ker(T_{\rm max} - \lambda) \big)^{-1}:
      H^{-1/2}(\Sigma; \mathbb{C}^4) \rightarrow H^{-1/2}(\Sigma; \mathbb{C}^4).
    \end{equation*}
    Moreover, it holds for any $\varphi \in H^{1/2}(\Sigma; \mathbb{C}^4)$ and $\psi \in H^{-1/2}(\Sigma; \mathbb{C}^4)$
    \begin{equation}\label{show}
       \big( \varphi, \widetilde{M}(\lambda) \psi \big)_{1/2 \times -1/2}
         = \big( M(\overline{\lambda}) \varphi, \psi \big)_{1/2 \times -1/2}.
    \end{equation}
    
    \item [\rm (iii)] The operators
    \begin{equation*}
    \widetilde{M}(\lambda)^2 - \frac{1}{4} I_4: H^{-1/2}(\Sigma; \mathbb{C}^4) \rightarrow H^{1/2}(\Sigma; \mathbb{C}^4)
  \end{equation*}
    are well-defined and bounded. In particular, $M(\lambda)^2 - \frac{1}{4}I_4$ is compact in $H^{1/2}(\Sigma; \mathbb{C}^4)$.
  \end{itemize}
\end{prop}
\begin{proof}
  Proposition~\ref{proposition_extension_gamma_Weyl_abstract} implies 
  the existence and continuity of $\widetilde{\gamma}(\lambda)$ and $\widetilde{M}(\lambda)$ in (i) and  (ii).
  In order to show \eqref{show} let $\varphi,\psi \in H^{1/2}(\Sigma; \mathbb{C}^4)$. Making use of \eqref{equation_diff_m} we find 
  \begin{equation*}
    \begin{split}
      \big( \varphi, \widetilde{M}(\lambda) \psi \big)_{1/2 \times -1/2}
          &= \big( (I_4 - \Delta_\Sigma)^{1/4} \varphi, (I_4 - \Delta_\Sigma)^{-1/4} \widetilde{M}(\lambda) \psi \big)_\Sigma \\
      &= \big( \varphi, M(\lambda) \psi \big)_\Sigma
          = \big( M(\overline{\lambda}) \varphi, \psi \big)_\Sigma \\
      &= \big( (I_4 - \Delta_\Sigma)^{1/4} M(\overline{\lambda})\varphi, (I_4 - \Delta_\Sigma)^{-1/4} \psi \big)_\Sigma \\
      &= \big( M(\overline{\lambda}) \varphi, \psi \big)_{1/2 \times -1/2}.
    \end{split}
  \end{equation*}
  A density argument yields \eqref{show}. 
  
  To prove item~(iii) we first consider the case $\lambda = 0$. 
  Note that equation~\eqref{equation_M_inv} and $(\alpha \cdot \nu)^2 = I_4$ imply
  \begin{equation} \label{equation_M_square}
      M(0)^2 - \frac{1}{4} I_4 = M(0) \alpha \cdot \nu \big( \alpha \cdot \nu M(0) + M(0) \alpha \cdot \nu \big).
  \end{equation}
  According to \cite[Proposition~2.8]{OV16} the operator 
  \begin{equation} \label{equation_def_A}
    \mathcal{A} := \alpha \cdot \nu M(0) + M(0) \alpha \cdot \nu: 
    H^{1/2}(\Sigma; \mathbb{C}^4) \rightarrow H^{1/2}(\Sigma; \mathbb{C}^4)
  \end{equation}
  admits a bounded extension $\widetilde{\mathcal{A}}: H^{-1/2}(\Sigma; \mathbb{C}^4) \rightarrow H^{1/2}(\Sigma; \mathbb{C}^4)$.
  This and \eqref{equation_M_square} show assertion~(iii) for $\lambda = 0$.
  
  Let now $\lambda \in \rho(A_0)$ be arbitrary. The identity~\eqref{equation_diff_m} yields for
  $\varphi \in H^{1/2}(\Sigma; \mathbb{C}^4)$
  \begin{equation} \label{equation_M_lambda_square}
    \begin{split}
      \left(M(\lambda)^2 - \frac{1}{4}\right) \varphi 
          &= \big( M(0) + \lambda \gamma(0)^* \gamma(\lambda) \big)^2 \varphi - \frac{1}{4} \varphi \\
      &= \left( M(0)^2 - \frac{1}{4} \right) \varphi + \lambda M(0) \gamma(0)^* \gamma(\lambda) \varphi
          + \lambda \gamma(0)^* \gamma(\lambda) M(0) \varphi \\
          &\qquad\qquad \qquad \qquad + \big(\lambda \gamma(0)^* \gamma(\lambda) \big)^2 \varphi.
    \end{split}
  \end{equation}
  From \eqref{equation_gamma_star} we get
  \begin{equation*}
    \ran \gamma(0)^* = \ran \big(\Gamma_1 A_0^{-1} \big) = H^{1/2}(\Sigma; \mathbb{C}^4).
  \end{equation*}
  Hence, the closed graph theorem implies that
  $\gamma(0)^*: L^2(\mathbb{R}^3; \mathbb{C}^4) \rightarrow H^{1/2}(\Sigma; \mathbb{C}^4)$ is continuous.
  Using item~(i) of this proposition we see that $\gamma(0)^* \gamma(\lambda)$ admits the continuous extension
  \begin{equation*}
    \gamma(0)^* \widetilde{\gamma}(\lambda): H^{-1/2}(\Sigma; \mathbb{C}^4) \rightarrow H^{1/2}(\Sigma; \mathbb{C}^4).
  \end{equation*}
  Moreover, since $M(0)$ has the continuous extension $\widetilde{M}(0)$ in $H^{-1/2}(\Sigma; \mathbb{C}^4)$ and
  $M(0)^2 - \frac{1}{4} I_4$ has a continuous extension from $H^{-1/2}(\Sigma; \mathbb{C}^4)$
  to $H^{1/2}(\Sigma; \mathbb{C}^4)$ by the previous considerations, equation~\eqref{equation_M_lambda_square} yields
  finally the statement of assertion~(iii) for all $\lambda \in \rho(A_0)$.
  
  Finally, since $\Sigma$ is compact, the embedding 
  $\iota: H^{1/2}(\Sigma; \mathbb{C}^4) \rightarrow H^{-1/2}(\Sigma; \mathbb{C}^4)$ is compact. Therefore, the mapping
  \begin{equation*}
    M(\lambda)^2 - \frac{1}{4} =  \left( \widetilde{M}(\lambda)^2 - \frac{1}{4} \right)\iota
  \end{equation*}
  is compact in $H^{1/2}(\Sigma; \mathbb{C}^4)$.
\end{proof}

Eventually, we provide a transformation of the quasi boundary triple 
from Theorem~\ref{theorem_triple} to an ordinary boundary triple.
This is an immediate consequence of Theorem~\ref{theorem_extended_boundary_triple_abstract}
for the special choice  $\iota_\pm = (I_4 - \Delta_\Sigma)^{\pm 1/4}$
and $\mu = 0 \in \rho(A_0)$,
but for the convenience of the reader we provide also a short direct proof.
In order to define the transformed boundary mappings recall that the direct sum decomposition
\begin{equation} \label{decomposition1}
  \dom T_{\rm max} = \dom A_0 \dot{+} \ker T_{\rm max}
\end{equation}
holds, as $0 \in \rho(A_0)$.

\begin{thm} \label{theorem_ordinary_triple}
  Let $S$ be the symmetric operator in \eqref{sss} with adjoint $S^* = T_{\rm max}$ in \eqref{def_T_m}. Moreover, let 
  $\{ L^2(\Sigma; \mathbb{C}^4), \Gamma_0, \Gamma_1 \}$ be the quasi boundary triple from Theorem~\ref{theorem_triple},
  let $\widetilde{\Gamma}_0$ be the extension of $\Gamma_0$ from Lemma~\ref{lemma_A_infty},
  and define the mappings $\Upsilon_0, \Upsilon_1: \dom T_{\rm max} \rightarrow L^2(\mathbb{R}^3; \mathbb{C}^4)$ by
  \begin{equation*}
    \Upsilon_0 f := (I_4 - \Delta_\Sigma)^{-1/4} \, \widetilde{\Gamma}_0 f\quad\text{and} \quad
    \Upsilon_1 f := (I_4 - \Delta_\Sigma)^{1/4} \Gamma_1 f_0,
  \end{equation*}
  where $f = f_0 + f_1 \in \dom A_0 \dot{+} \ker T_{\rm max} = \dom T_{\rm max}$.
  Then $\{ L^2(\Sigma; \mathbb{C}^4), \Upsilon_0, \Upsilon_1 \}$ is an ordinary boundary triple for $S^* = T_{\rm max}$ 
  with $S^* \upharpoonright \ker\Upsilon_0=T \upharpoonright \ker\Gamma_0 = A_0$.
\end{thm}
\begin{proof}
  First we verify that Green's identity is true. Assume that $f, g \in \dom T \subset \dom T_{\rm max}$
  and decompose these functions with respect to~\eqref{decomposition1} as 
  $f = f_0 + f_1$ and $g = g_0 + g_1$ with $f_0, g_0 \in \dom A_0$ and $f_1, g_1 \in \ker T$.
  Using the self-adjointness of the free Dirac operator $A_0$ and $T f_1 = T g_1 = 0$ we deduce then
  \begin{equation*}
    \begin{split}
      (T f, g)_{\mathbb{R}^3} - (f, T g)_{\mathbb{R}^3} 
          &= (A_0 f_0, g_0)_{\mathbb{R}^3} + (T f_0, g_1)_{\mathbb{R}^3} 
          - (f_0, A_0 g_0)_{\mathbb{R}^3} - (f_1, T g_0)_{\mathbb{R}^3} \\
      &= (T f_0, g_1)_{\mathbb{R}^3} - (f_0, T g_1)_{\mathbb{R}^3} 
            + (T f_1, g_0)_{\mathbb{R}^3} - (f_1, T g_0)_{\mathbb{R}^3}.
    \end{split}
  \end{equation*}
  Employing now Theorem~\ref{theorem_triple} and $\Gamma_0 f_0 = \Gamma_0 g_0 = 0$, as $f_0, g_0 \in \dom A_0 = \ker \Gamma_0$,
  we get
  \begin{equation*}
    \begin{split}
      (T f, g)_{\mathbb{R}^3} &- (f, T g)_{\mathbb{R}^3} \\
      &= (\Gamma_1 f_0, \Gamma_0 g_1)_{\Sigma} - (\Gamma_0 f_0, \Gamma_1 g_1)_{\Sigma}
          + (\Gamma_1 f_1, \Gamma_0 g_0)_{\Sigma} - (\Gamma_0 f_1, \Gamma_1 g_0)_{\Sigma} \\
      &= (\Gamma_1 f_0, \Gamma_0 g)_{\Sigma} - (\Gamma_0 f, \Gamma_1 g_0)_{\Sigma}.
    \end{split}
  \end{equation*}
  The self-adjointness of $(I_4 - \Delta_\Sigma)^{\pm 1/4}$ implies eventually
  \begin{equation*}
    \begin{split}
      (T f, g)_{\mathbb{R}^3} - (f, T g)_{\mathbb{R}^3} 
        &= \big((I_4 - \Delta_\Sigma)^{1/4} \Gamma_1 f_0, (I_4 - \Delta_\Sigma)^{-1/4} \Gamma_0 g\big)_{\Sigma} \\
            &\qquad - \big((I_4 - \Delta_\Sigma)^{-1/4} \Gamma_0 f, (I_4 - \Delta_\Sigma)^{1/4} \Gamma_1 g_0\big)_{\Sigma}\\
        &= (\Upsilon_1 f, \Upsilon_0 g)_{\Sigma} - (\Upsilon_0 f, \Upsilon_1 g)_{\Sigma}.
    \end{split}
  \end{equation*}
  Since $C^\infty(\overline{\Omega_+}; \mathbb{C}^4) \oplus C^\infty(\overline{\Omega_+}; \mathbb{C}^4) \subset \dom T$
  is dense in $\dom T_{\rm max}$ by Lemma~\ref{lemma_core_T} and $\widetilde{\Gamma}_0, \Gamma_1$ are continuous with 
  respect to the graph norm by Lemma~\ref{lemma_A_infty} we conclude that Green's identity holds for all 
  $f, g \in \dom T_{\rm max}$.
  
  To see that $(\Upsilon_0, \Upsilon_1)$ is surjective let $\varphi, \psi \in L^2(\Sigma; \mathbb{C}^4)$
  arbitrary, but fixed. Since $\ran \widetilde{\Gamma}_0 = H^{-1/2}(\Sigma; \mathbb{C}^4)$ 
  by Lemma~\ref{lemma_A_infty} there exists $g \in \dom \widetilde{\Gamma}_0 = \dom T_{\rm max}$ such that
  \begin{equation*}
    \Upsilon_0 g = (I_4 - \Delta_\Sigma)^{-1/4} \widetilde{\Gamma}_0 g = \varphi.
  \end{equation*}
  Next, choose some $h \in \dom A_0$ which satisfies
  \begin{equation*}
    \Upsilon_1 h = (I_4 - \Delta_\Sigma)^{1/4} \Gamma_1 h = \psi - \Upsilon_1 g.
  \end{equation*}
  Since $h \in \dom A_0 = \ker \Gamma_0$ we have, in particular, $\Upsilon_0 h = 0$.
  Thus, the function $f := g + h \in \dom T_{\rm max}$ fulfills $\Upsilon_0 f = \varphi$ and $\Upsilon_1 f = \psi$,
  which shows that $(\Upsilon_0, \Upsilon_1)$ is surjective.
  
  It remains to show that $T_{\rm max} \upharpoonright \ker \Upsilon_0$ is self-adjoint. 
  With the help of Green's identity, which is already proved, it is easy to see that 
  $T_{\rm max} \upharpoonright \ker \Upsilon_0$ is symmetric. Moreover, the self-adjoint free Dirac operator 
  is contained in $T_{\rm max} \upharpoonright \ker \Upsilon_0$. Thus, these operators must coincide.
  Therefore, the triple $\{ L^2(\Sigma; \mathbb{C}^4), \Upsilon_0, \Upsilon_1 \}$ fulfills all conditions to
  be an ordinary boundary triple for $S^* = T_{\rm max}$ in the sense of Definition~\ref{qbtdef} and the proof of this 
  theorem is finished.
\end{proof}

\section{Dirac operators with electrostatic $\delta$-shell interactions} \label{section_Dirac_delta}

In this section we define and investigate Dirac operators $A_{\eta}$ with $\delta$-shell interactions supported on the 
closed and bounded $C^2$-surface $\Sigma \subset \mathbb{R}^3$ with interaction strength $\eta \in \mathbb{R}$. 
In particular, we treat the case of the critical interaction strength $\eta = \pm 2$, 
for which self-adjointness and other spectral properties of $A_{\pm 2}$ were not obtained so far. The strategy is as follows:
Using the quasi boundary 
triple from Theorem~\ref{theorem_triple} and the transformed ordinary boundary triple from 
Theorem~\ref{theorem_ordinary_triple} with the corresponding transformed parameter 
$\Theta_1(\pm 2)$, cf. \eqref{transformed_parameter},
we identify the closure of $A_{\pm 2}$ with the closure of  $\Theta_1(\pm 2)$, which turns out to be self-adjoint in $L^2(\Sigma; \mathbb{C}^4)$.
Making use of the corresponding Weyl functions and by constructing suitable singular sequences 
we prove some spectral results for $\overline{A_{\pm 2}}$ 
in Theorem~\ref{proposition_properties_A_pm2}.

We start with the definition of Dirac operators with an electrostatic $\delta$-shell interaction 
of constant strength. 

\begin{definition} \label{definition_A_eta}
  Let $\{ L^2(\Sigma; \mathbb{C}^4), \Gamma_0, \Gamma_1 \}$ be the quasi boundary triple 
  from Theorem~\ref{theorem_triple} for $S^* = T_{\rm max}=\overline{T}$ with $T$ in \eqref{def_T} and let $\eta \in \mathbb{R}$. 
  Then the Dirac operator $A_\eta$ with  an electrostatic $\delta$-shell interaction 
  of strength $\eta$ is defined by
  \begin{equation} \label{def_A_eta}
    A_\eta := T \upharpoonright \ker(\Gamma_0 + \eta \Gamma_1),
  \end{equation}
  that is, 
  \begin{equation*}
    \begin{split}
      A_\eta f &= (-i \alpha \cdot \nabla + m \beta) f_+ \oplus (-i \alpha \cdot \nabla + m \beta) f_-, \\
      \dom A_\eta &= \big\{ f = f_+ \oplus f_- \in \dom T: 
          \tfrac{\eta}{2} (f_+|_\Sigma + f_-|_\Sigma) =- i \alpha \cdot \nu (f_+|_\Sigma - f_-|_\Sigma) \big\}.
    \end{split}
  \end{equation*}
\end{definition}

It follows immediately from \eqref{def_A_eta} and Green's identity
that $A_\eta$ is symmetric for any $\eta \in \mathbb{R}$.
In the following proposition we prove that $A_\eta$ is self-adjoint for $\eta \neq \pm 2$; 
similar results have been obtained in~\cite{AMV14, BEHL16_2},
but the approach used here also yields an additional regularity result for the  functions in $\dom A_\eta$.

\begin{prop} \label{proposition_A_eta_self_adjoint_nice}
  Let $\eta \in \mathbb{R} \setminus \{ \pm 2 \}$ and let $A_\eta$ be defined as in Definition~\ref{definition_A_eta}.
  Then $A_\eta$ is self-adjoint and $\dom A_\eta\subset H^1(\mathbb{R}^3\setminus \Sigma; \mathbb{C}^4)$.
\end{prop}
\begin{proof}
  Let $\eta \in \mathbb{R} \setminus \{ \pm 2 \}$ be fixed and assume $\eta\not=0$ (note that $A_0$ is the self-adjoint free Dirac operator).
  In order to show that the symmetric operator $A_\eta$ is self-adjoint we verify 
  $\ran(A_\eta - \lambda) = L^2(\mathbb{R}^3; \mathbb{C}^4)$ for all $\lambda \in \mathbb{C} \setminus \mathbb{R}$.
  For $\lambda \in \mathbb{C} \setminus \mathbb{R}$ we have $\lambda\not\in\sigma_{\rm p}(A_\eta)$ since $A_\eta$ is symmetric.
  Hence, by Theorem~\ref{theorem_Krein_abstract}~(ii) the operator $A_\eta - \lambda$ is surjective if 
  $\ran \gamma( \overline{\lambda} )^* \subset \ran\big( -\frac{1}{\eta} I_4 - M(\lambda) \big)$.
  Observe first that
  \begin{equation*}
    \ran \gamma ( \overline{\lambda})^* = \ran\big( \Gamma_1 (A_0 - \lambda)^{-1} \big) 
    = H^{1/2}(\Sigma; \mathbb{C}^4)
  \end{equation*}
  by \eqref{equation_gamma_star}. To show that $H^{1/2}(\Sigma; \mathbb{C}^4) \subset \ran\big(-\frac{1}{\eta} I_4 - M(\lambda)\big)$
  we note that
  \begin{equation}\label{plm}
   \left(-\frac{1}{\eta} I_4 - M(\lambda) \right) 
            \left(-\frac{1}{\eta} I_4 + M(\lambda) \right)=\left(\frac{1}{\eta^2} - \frac{1}{4} \right) I_4 + K(\lambda)
  \end{equation}
  with $K(\lambda) := \frac{1}{4} I_4 - M(\lambda)^2$.
  By Proposition~\ref{proposition_extension_gamma_Weyl}~(iii) the operator 
  $K(\lambda)$ is compact in $H^{1/2}(\Sigma; \mathbb{C}^4)$.
  Moreover \eqref{plm} is an injective operator as otherwise one of the symmetric 
  operators $A_{\pm \eta}$ would have the non-real eigenvalue $\lambda$; cf. Theorem~\ref{theorem_Krein_abstract}~(i).
  Thus, Fredholm's alternative and \eqref{plm} yield  
  \begin{equation*} 
      H^{1/2}(\Sigma; \mathbb{C}^4)= \ran \left(\left( \frac{1}{\eta^2} - \frac{1}{4} \right) I_4 + K(\lambda) \right)\subset  
      \ran \left( -\frac{1}{\eta} I_4 - M(\lambda) \right).
  \end{equation*}
  From this and the above considerations it follows that $A_\eta$ is self-adjoint for $\eta \in \mathbb{R} \setminus \{ \pm 2 \}$. The inclusion
  $\dom A_\eta\subset H^1(\mathbb{R}^3\setminus \Sigma; \mathbb{C}^4)$ is clear from \eqref{def_A_eta}.
\end{proof}

Now we turn our attention to the critical case $\eta=\pm 2$ in Definition~\ref{definition_A_eta}. 

\begin{prop}\label{propi}
The operators
\begin{equation}\label{xxx}
 A_{\pm 2}=T\upharpoonright \ker(\Gamma_0 \pm 2 \Gamma_1)=T\upharpoonright \ker\bigl(\Gamma_1\pm\tfrac{1}{2}\Gamma_0\bigr)
\end{equation}  
are symmetric in $L^2(\mathbb{R}^3; \mathbb{C}^4)$ but not self-adjoint.
\end{prop}

\begin{proof} 
First, it follows from Green's identity (Definition~\ref{qbtdef}~(i)) that $A_{\pm 2}$ are both symmetric.
  Now assume that $A_2$ is self-adjoint; the same argument applies to $A_{-2}$.
  Then $\ran(A_2 - \lambda) = L^2(\mathbb{R}^3; \mathbb{C}^4)$ for any 
  $\lambda \in \mathbb{C} \setminus \mathbb{R}$ and hence Theorem~\ref{theorem_Krein_abstract}~(ii)
  yields
  \begin{equation} \label{equation_range_M_p}
    \ran \gamma( \overline{\lambda} )^* = H^{1/2}(\Sigma; \mathbb{C}^4) 
        \subset \ran\left( -\frac{1}{2} I_4 - M(\lambda) \right).
  \end{equation}
  Since $\lambda\not\in\sigma_{\rm p}(A_2)$ it follows that $-\frac{1}{2}I_4 - M(\lambda)$ 
  is bijective in $H^{1/2}(\Sigma; \mathbb{C}^4)$.
  
  We claim that \eqref{equation_range_M_p} also implies 
  $\ran\big( \frac{1}{2} I_4 - M(\lambda) \big) = H^{1/2}(\Sigma; \mathbb{C}^4)$.
  In fact, for $\varphi \in H^{1/2}(\Sigma; \mathbb{C}^4)$ we have 
  $-2 M(\lambda) \, \alpha \cdot \nu \, \varphi \in H^{1/2}(\Sigma; \mathbb{C}^4)$ since
  $\Sigma$ is $C^2$-smooth and $M(\lambda)$ maps $H^{1/2}(\Sigma; \mathbb{C}^4)$ into $H^{1/2}(\Sigma; \mathbb{C}^4)$; cf. 
  Proposition~\ref{proposition_gamma_Weyl}~(ii).
  By \eqref{equation_range_M_p} there exists $\psi \in H^{1/2}(\Sigma; \mathbb{C}^4)$ such that 
  \begin{equation*}
    -2 M(\lambda) \, \alpha \cdot \nu \, \varphi = \left( -\frac{1}{2} - M(\lambda) \right) \psi.
  \end{equation*}
  Applying on both sides $M(\lambda) \alpha \cdot \nu$ and using $(\alpha \cdot \nu)^2 = I_4$ and 
  $(M(\lambda) \alpha \cdot \nu)^2 = -\frac{1}{4} I_4$ (see~\eqref{equation_M_inv}) 
  we find
  \begin{equation*}
    \frac{1}{2} \varphi = -2 (M(\lambda) \alpha \cdot \nu)^2 \varphi
    = \left( -\frac{1}{2} M(\lambda) \alpha \cdot \nu - M(\lambda) \alpha \cdot \nu M(\lambda) (\alpha \cdot \nu)^2 \right)
    \psi,
  \end{equation*}
  which is equivalent to $\varphi = \big( \frac{1}{2} - M(\lambda) \big)\, \alpha \cdot \nu\, \psi$,
  i.e. $\varphi \in \ran \big( \frac{1}{2} I_4 - M(\lambda) \big)$. We have shown 
  $\ran \big( \frac{1}{2} I_4 - M(\lambda) \big) = H^{1/2}(\Sigma, \mathbb{C}^4)$, and as 
  $\lambda \not\in \sigma_{\rm p}(A_{-2})$ it follows that $\frac{1}{2} I_4 - M(\lambda)$ 
  is bijective in $H^{1/2}(\Sigma; \mathbb{C}^4)$.
  
  Since $-\frac{1}{2} I_4 - M(\lambda)$ and $\frac{1}{2} I_4 - M(\lambda)$ are both bijective on $H^{1/2}(\Sigma; \mathbb{C}^4)$ 
  also $M(\lambda)^2 - \frac{1}{4} I_4$ is bijective on $H^{1/2}(\Sigma; \mathbb{C}^4)$.
  On the other hand, $M(\lambda)^2 - \frac{1}{4} I_4$ is compact in $H^{1/2}(\Sigma; \mathbb{C}^4)$ by Proposition~\ref{proposition_extension_gamma_Weyl}~(iii).
  Hence, this operator can not be surjective; a contradiction.
  Therefore $A_2$ is not self-adjoint in $L^2(\mathbb{R}^3; \mathbb{C}^4)$.
\end{proof}

In the following we complement Proposition~\ref{propi}, show that $A_{\pm 2}$ in \eqref{xxx} is essentially self-adjoint
and determine the closure $\overline{A_{\pm 2} }$.
For this we shall also  consider  the ordinary boundary triple 
$\{ L^2(\Sigma; \mathbb{C}^4), \Upsilon_0, \Upsilon_1 \}$ in Theorem~\ref{theorem_ordinary_triple} with $\gamma$-field 
\begin{equation}\label{gammm}
 \beta(\lambda)=\widetilde\gamma(\lambda) (I_4 - \Delta_\Sigma)^{1/4},\qquad\lambda\in\rho(A_0),
\end{equation}
and Weyl function
\begin{equation} \label{def_M_OBT}
 \mathcal M(\lambda)=(I_4 - \Delta_\Sigma)^{1/4}\bigl(\widetilde M(\lambda)-\widetilde M(0) \bigr)(I_4 - \Delta_\Sigma)^{1/4},\qquad\lambda\in\rho(A_0).
\end{equation}
The corresponding parameter (here $\vartheta=\mp\tfrac{1}{2}$ by \eqref{xxx})
in Corollary~\ref{corollary_self_adjoint_transformed_triple}  is then given by
\begin{equation*}
  \begin{split}
    \Theta_1(\pm 2) &:= -(I_4 - \Delta_\Sigma)^{1/4} \left( \pm \frac{1}{2} + M(0) \right) (I_4 - \Delta_\Sigma)^{1/4}, \\
    \dom \Theta_1(\pm 2) &:= H^1(\Sigma; \mathbb{C}^4),
  \end{split}
\end{equation*}
that is, 
\begin{equation}\label{xxxx}
A_{\pm 2}= T_{\rm max}\upharpoonright\ker\bigl(\Upsilon_1 - \Theta_1(\pm 2)\Upsilon_0\bigr).
\end{equation}
Clearly, $\Theta_1(\pm 2)$
is symmetric in $L^2(\Sigma; \mathbb{C}^4)$.
Our next goal is to prove that $\Theta_1(\pm 2)$ is essentially self-adjoint and
that $\overline{\Theta_1(\pm 2)}$ coincides with the maximal operator 
\begin{equation} \label{def_theta_max}
  \begin{split}
    \Theta_{\rm max}(\pm 2) \varphi &:= -(I_4 - \Delta_\Sigma)^{1/4} 
       \left(\pm \frac{1}{2} + \widetilde{M}(0) \right) (I_4 - \Delta_\Sigma)^{1/4} \varphi,\\
    \dom \Theta_{\rm max}(\pm 2) &:= \left\{ \varphi \in L^2(\Sigma; \mathbb{C}^4):\! \left(
        \pm \frac{1}{2} + \widetilde{M}(0) \right) (I_4 - \Delta_\Sigma)^{1/4} \varphi
        \in H^{1/2}(\Sigma; \mathbb{C}^4) \right\}\!.
  \end{split}
\end{equation}

\begin{lem} \label{proposition_Theta_1}
  The operator $\Theta_1(\pm 2)$ is essentially self-adjoint in $L^2(\Sigma; \mathbb{C}^4)$ and 
  $\overline{\Theta_1(\pm 2)} = \Theta_{\rm max}(\pm 2)$.
  In particular, $\Theta_{\rm max}(\pm 2)$ is self-adjoint.
\end{lem}
\begin{proof}
  We  prove the statement for  $\eta = 2$, the case $\eta = -2$ is analogous. 
  For the convenience of the reader we divide the proof in three steps.
  
  \noindent
  {\it Step 1.} We check first that $\Theta_{\rm{max}}(2)$ is closed. For this let
  $(\varphi_n) \subset \dom \Theta_{\rm max}(2)$ such that $\varphi_n \rightarrow \varphi \in L^2(\Sigma; \mathbb{C}^4)$
  and $\Theta_{\rm max}(2) \varphi_n \rightarrow \psi \in L^2(\Sigma; \mathbb{C}^4)$ as $n \rightarrow \infty$. Since 
  $(I_4 - \Delta_\Sigma)^{-1/4}: L^2(\Sigma; \mathbb{C}^4) \rightarrow H^{1/2}(\Sigma; \mathbb{C}^4)$ 
  is an isomorphism we find 
  \begin{equation} \label{equation_convergence1}
    -\left( \frac{1}{2} + \widetilde{M}(0) \right) (I_4 - \Delta_\Sigma)^{1/4} \varphi_n 
        \rightarrow (I_4 - \Delta_\Sigma)^{-1/4} \psi, \quad n \rightarrow \infty,
  \end{equation}
  with respect to the $H^{1/2}$-norm, and hence also with respect to the $H^{-1/2}$-norm.
  On the other hand, since $(I_4 - \Delta_\Sigma)^{1/4}: L^2(\Sigma; \mathbb{C}^4) \rightarrow H^{-1/2}(\Sigma; \mathbb{C}^4)$
  and $\widetilde{M}(0)$ is continuous in $H^{-1/2}(\Sigma; \mathbb{C}^4)$ by 
  Proposition~\ref{proposition_extension_gamma_Weyl}~(ii) we obtain
  \begin{equation*}
    -\left( \frac{1}{2} + \widetilde{M}(0) \right) (I_4 - \Delta_\Sigma)^{1/4} \varphi_n 
        \rightarrow -\left( \frac{1}{2} + \widetilde{M}(0) \right) (I_4 - \Delta_\Sigma)^{1/4} \varphi, 
        \quad n \rightarrow \infty,
  \end{equation*}
  with respect to the $H^{-1/2}$-norm. Combining this with \eqref{equation_convergence1} the last observation leads to
  \begin{equation*}
    -\left( \frac{1}{2} + \widetilde{M}(0) \right) (I_4 - \Delta_\Sigma)^{1/4} \varphi 
        = (I_4 - \Delta_\Sigma)^{-1/4} \psi \in H^{1/2}(\Sigma; \mathbb{C}^4).
  \end{equation*}
  Therefore, $\varphi \in \dom \Theta_{\rm max}(2)$ and $\Theta_{\rm max}(2) \varphi = \psi$.
  Thus, $\Theta_{\rm{max}}(2)$ is closed.

  \noindent
  {\it Step 2.}
  Let us now show the inclusion
  \begin{equation}\label{mkmk}
   \Theta_{\rm max}(2) \subset \overline{\Theta_1(2)};
  \end{equation}
  together with $\Theta_1(2) \subset \Theta_{\rm max}(2)$ and $\Theta_{\rm{max}}(2)$ closed from {\it Step 1} this  yields 
  \begin{equation}\label{qqq}
  \Theta_{\rm max}(2)= \overline{\Theta_1(2)}. 
  \end{equation}
  To prove \eqref{mkmk} let $\varphi \in \dom \Theta_{\rm max}(2)$ and choose a sequence $(\psi_n) \subset H^1(\Sigma; \mathbb{C}^4)$
  such that $\psi_n \rightarrow \varphi$ in $L^2(\Sigma; \mathbb{C}^4)$ as $n \rightarrow \infty$.
  We define
  \begin{equation*}
    \varphi_n := \varphi 
        + (I_4 - \Delta_\Sigma)^{-1/4} \left( \widetilde{M}(0) - \frac{1}{2} \right) (I_4 - \Delta_\Sigma)^{1/4} (\varphi - \psi_n).
  \end{equation*}
  It follows from 
  \begin{equation*}
    \begin{split}
      \varphi_n 
      &= (I_4 - \Delta_\Sigma)^{-1/4} \left( \frac{1}{2}+ \widetilde{M}(0) \right) (I_4 - \Delta_\Sigma)^{1/4} \varphi \\
      &\qquad\qquad + (I_4 - \Delta_\Sigma)^{-1/4} \left( \frac{1}{2} - M(0) \right) (I_4 - \Delta_\Sigma)^{1/4} \psi_n,
    \end{split}
  \end{equation*}
  $\big( \frac{1}{2} + \widetilde{M}(0) \big) (I_4 - \Delta_\Sigma)^{1/4} \varphi \in H^{1/2}(\Sigma; \mathbb{C}^4)$ for $\varphi \in \dom \Theta_{\rm max}(2)$, 
  Proposition~\ref{proposition_gamma_Weyl} and the mapping properties of $(I_4 - \Delta_\Sigma)^{-1/4}$ that
  $\varphi_n\in H^1(\Sigma; \mathbb{C}^4) = \dom \Theta_1(2)$.
  Moreover, since $\widetilde{M}(0)$ is continuous in $H^{-1/2}(\Sigma; \mathbb{C}^4)$ we obtain 
  \begin{equation*}
    \varphi_n - \varphi 
        = (I_4 - \Delta_\Sigma)^{-1/4} \left( \widetilde{M}(0) - \frac{1}{2} \right) (I_4 - \Delta_\Sigma)^{1/4} (\varphi - \psi_n)
        \rightarrow 0, \quad n \rightarrow \infty,
  \end{equation*}
  in $L^2(\Sigma; \mathbb{C}^4)$.
  Finally, since $\widetilde{M}(0)^2 - \frac{1}{4} I_4: H^{-1/2}(\Sigma; \mathbb{C}^4) \rightarrow H^{1/2}(\Sigma; \mathbb{C}^4)$
  is continuous by Proposition~\ref{proposition_extension_gamma_Weyl}~(iii) we have
  \begin{equation*}
    \Theta_{\rm max}(2) (\varphi - \varphi_n) = (I_4 - \Delta_\Sigma)^{1/4} \left( \widetilde{M}(0)^2 - \frac{1}{4} \right) 
        (I_4 - \Delta_\Sigma)^{1/4} (\varphi - \psi_n) \rightarrow 0, \quad n \rightarrow \infty,
  \end{equation*}
  in $L^2(\Sigma; \mathbb{C}^4)$. In particular, as $\Theta_1(2)\subset \Theta_{\rm max}(2)$ 
  we have $\Theta_1(2)\varphi_n\rightarrow \Theta_{\rm max}(2)\varphi$ as $n \rightarrow \infty$,
and hence $\varphi \in \dom \overline{\Theta_1(2)}$ and $\overline{\Theta_1(2)}\varphi=\Theta_{\rm max}(2)\varphi$, i.e. \eqref{mkmk} holds. 
  
  \noindent
  {\it Step 3.} Since $\Theta_1(2)$ is symmetric it follows from \eqref{qqq} that $\overline{\Theta_1(2)}=\Theta_{\rm max}(2)$ is a symmetric operator.
  It is also clear from \eqref{qqq} that $\Theta_1(2)^*=\Theta_{\rm max}(2)^*$.
  In order to conclude that $\Theta_{\rm max}(2)$ is self-adjoint it suffices to show the inclusion
  \begin{equation}\label{hhh}
   \Theta_1(2)^* \subset \Theta_{\rm max}(2).
  \end{equation}
  For this
  let $\psi \in \dom \Theta_1(2)^*$ and $\varphi \in H^1(\Sigma; \mathbb{C}^4) = \dom \Theta_1(2)$.
  Making  use of \eqref{show} we compute
  \begin{equation*}
    \begin{split}
      \big( \Theta_1(2)^* \psi, \varphi \big)_\Sigma
          &= \big( \psi, \Theta_1(2) \varphi \big)_\Sigma \\
      &= \Big( \psi, -(I_4 - \Delta_\Sigma)^{1/4} 
          \left(\tfrac{1}{2} + M(0) \right) (I_4 - \Delta_\Sigma)^{1/4} \varphi \Big)_\Sigma \\
      &= \Big( (I_4 - \Delta_\Sigma)^{1/4} \psi,  
          -\left(\tfrac{1}{2} + M(0) \right) (I_4 - \Delta_\Sigma)^{1/4} \varphi \Big)_{-1/2 \times 1/2} \\
      &= \Big( -\left(\tfrac{1}{2} + \widetilde{M}(0) \right) (I_4 - \Delta_\Sigma)^{1/4} \psi,  
           (I_4 - \Delta_\Sigma)^{1/4} \varphi \Big)_{-1/2 \times 1/2} \\
       &= \Big( -(I_4 - \Delta_\Sigma)^{-1/4} \left(\tfrac{1}{2} + \widetilde{M}(0) \right) (I_4 - \Delta_\Sigma)^{1/4} \psi,  
           (I_4 - \Delta_\Sigma)^{1/2} \varphi \Big)_\Sigma.
    \end{split}
  \end{equation*}
  Since this is true for any 
  $\varphi \in H^1(\Sigma; \mathbb{C}^4) = \dom(I_4 - \Delta_\Sigma)^{1/2}$
  we conclude 
  $$-(I_4 - \Delta_\Sigma)^{-1/4} \left(\tfrac{1}{2} + \widetilde{M}(0) \right) (I_4 - \Delta_\Sigma)^{1/4} \psi
  \in \dom\big((I_4 - \Delta_\Sigma)^{1/2} \big)^* = H^1(\Sigma; \mathbb{C}^4)$$ and 
  using the self-adjointness of $(I_4 - \Delta_\Sigma)^{1/2}$ we find 
  \begin{equation*}
    -(I_4 - \Delta_\Sigma)^{1/2} \Big((I_4 - \Delta_\Sigma)^{-1/4} 
        \left(\tfrac{1}{2} + \widetilde{M}(0) \right) (I_4 - \Delta_\Sigma)^{1/4} \psi \Big)
        = \Theta_1(2)^* \psi.
  \end{equation*}
  This implies  $\psi \in \dom \Theta_{\rm max}(2)$ and $\Theta_1(2)^* \psi = \Theta_{\rm max}(2) \psi$, and hence \eqref{hhh} holds.
\end{proof}

In the following theorem we conclude that the symmetric operator $A_{\pm 2}$ is essentially self-adjoint
and provide the domain of its closure, which is the proper self-adjoint realization of the Dirac
operator with an electrostatic $\delta$-shell interaction of strength~$\pm 2$.
For that recall the definitions of the maximal Dirac operator $T_{\rm max}$ from \eqref{def_T_m},
the extended boundary mappings $\widetilde{\Gamma}_0, \widetilde{\Gamma}_1$ in
Lemma~\ref{lemma_A_infty}, and of the ordinary boundary triple 
$\{ L^2(\Sigma; \mathbb{C}^4), \Upsilon_0, \Upsilon_1\}$  in 
Theorem~\ref{theorem_ordinary_triple}.

\begin{thm} \label{theorem_self_adjoint_bad}
  The operator $A_{\pm 2}$ in~\eqref{def_A_eta} is essentially self-adjoint in $L^2(\mathbb{R}^3; \mathbb{C}^4)$ and
  the self-adjoint closure is given by 
  \begin{equation} \label{def_A_2_closure}
    \overline{A_{\pm 2}} = T_{\rm max} \upharpoonright \ker\big( \Upsilon_1 - \Theta_{\rm max}(\pm 2) \Upsilon_0 \big)
        = T_{\rm max} \upharpoonright \ker\big( \widetilde{\Gamma}_0 \pm 2 \widetilde{\Gamma}_1 \big).
  \end{equation}
  Furthermore, $A_{\pm 2} \subsetneq \overline{A_{\pm 2}}$ and $\dom \overline{A_{\pm 2}} \not\subset H^1(\mathbb{R}^3 \setminus \Sigma; \mathbb{C}^4)$.
\end{thm}

\begin{proof}
It follows from Lemma~\ref{proposition_Theta_1} and \eqref{xxxx} that $A_{\pm 2}$ is 
  essentially self-adjoint; cf. Corollary~\ref{corollary_self_adjoint_transformed_triple}. Furthermore, since 
  $\{ L^2(\Sigma; \mathbb{C}^4), \Upsilon_0, \Upsilon_1 \}$ is an ordinary boundary triple 
  the closure $\overline{A_{\pm 2}}$ corresponds to
  the closure of the parameter $\Theta_1(\pm 2)$, that is,
 $$
 \overline{A_{\pm 2}} = T_{\rm max} \upharpoonright \ker\big( \Upsilon_1 - \Theta_{\rm max}(\pm 2) \Upsilon_0 \big),
 $$
 and $\overline{A_{\pm 2}}$ is self-adjoint in $L^2(\mathbb{R}^3; \mathbb{C}^4)$; cf. Lemma~\ref{proposition_Theta_1} and 
 Corollary~\ref{corollary_self_adjoint_transformed_triple}. The second equality in \eqref{def_A_2_closure} can be checked directly and also follows 
 from \cite[Corollary~3.8]{BM14}. The last assertions are consequences of Proposition~\ref{propi}.
\end{proof}

\begin{remark}
The boundary condition $\pm 2 \widetilde{\Gamma}_1 f = -\widetilde{\Gamma}_0 f$ for $f \in \dom T_{\rm max}$ 
in Theorem~\ref{theorem_self_adjoint_bad} 
is understood in $H^{-1/2}(\Sigma; \mathbb{C}^4)$ and with traces interpreted in $H^{-1/2}(\Sigma; \mathbb{C}^4)$ 
(cf. \cite[Proposition~2.1]{OV16}) it has the more explicit form
  \begin{equation*}
   \pm  (f_+|_\Sigma + f_-|_\Sigma) =- i \alpha \cdot \nu (f_+|_\Sigma - f_-|_\Sigma),\qquad f \in \dom T_{\rm max},
  \end{equation*} 
  which is in accordance with Definition~\ref{definition_A_eta}.
\end{remark}

In the next theorem we discuss some spectral properties of the self-adjoint operator $\overline{A_{\pm 2}}$; 
the results complement 
those for $A_\eta$, $\eta\not=\pm 2$, from Theorem~\ref{theorem_spectrum_noncritical}.
We point out that, in contrast to the non-critical case $\eta\not=\pm 2$, in the critical case
$\eta=\pm 2$
the interval $(-m,m)$ may contain essential spectrum, see also Theorem~\ref{theorem_essential_spectrum} below.

\begin{thm} \label{proposition_properties_A_pm2} 
  The following assertions hold for the self-adjoint operators $\overline{A_{\pm 2}}$:
  \begin{itemize}
    \item[\rm (i)] $(-\infty, -m] \cup [m, \infty) \subset \sigma_{\rm ess}(\overline{A_{\pm 2}})$;
    \item[\rm (ii)] $\lambda \in (-m,m)\cap \sigma_{\rm p}(\overline{A_{\pm 2}})$ if and only if 
    $0 \in \sigma_{\rm p}\big( 1 \pm 2 \widetilde{M}(\lambda) \big)$;
    \item[\rm (iii)] $\sigma_{\rm disc}(\overline{A_2}) = \sigma_{\rm disc}(\overline{A_{-2}})$
    and $\sigma_{\rm ess}(\overline{A_2}) = \sigma_{\rm ess}(\overline{A_{-2}})$;
    \item[\rm (iv)] For $\lambda \in \rho(\overline{A_{\pm 2}})$ it holds that
    \begin{equation*}
      (\overline{A_{\pm 2}} - \lambda)^{-1} = (A_0 - \lambda)^{-1} \mp \widetilde{\gamma}(\lambda)
          \big( 1 \pm 2 \widetilde{M}(\lambda) \big)^{-1} 2 \gamma(\overline{\lambda})^*.
    \end{equation*}
  \end{itemize}
\end{thm}
\begin{proof}
  (i) We verify the inclusion $(-\infty, -m] \cup [m, \infty)\subset\sigma_{\rm ess}(\overline{A_2})$; the 
  inclusion
  for the strength $-2$ can be verified in the same way. 
  For $\lambda \in (-\infty, -m] \cup [m, \infty)$ fixed we construct a singular sequence as follows.
  First of all, since $\Sigma$ is compact we can choose $R>0$ such that $\Sigma \subset B(0, R)$. Next, 
  let $\chi \in C_c^\infty(\mathbb{R})$ be a cutoff function satisfying
  $\chi(r) = 1$ for $|r| < \frac{1}{2}$ and $\chi(r) = 0$ for $|r| > 1$
  and set $x_n := (R + n^2) e_1$, where $e_1=(1, 0, 0)^\top$. We define
  \begin{equation} \label{singular_sequence}
    \psi_n^\lambda(x) := \frac{1}{n^{3/2}} \chi\left( \frac{1}{n} |x - x_n| \right) e^{i \sqrt{\lambda^2 - m^2} x\cdot e_1}
    \left( \sqrt{\lambda^2 - m^2} \alpha_1 + m \beta + \lambda \right) \zeta,
  \end{equation}
  where $\zeta \in \mathbb{C}^4$ is chosen such that 
  $\left( \sqrt{\lambda^2 - m^2} \alpha_1 + m \beta + \lambda \right) \zeta \neq 0$.
  By construction we have $\text{supp}\, \psi_n^\lambda \cap \Sigma = \emptyset$ and thus
  $\psi_n^\lambda \in \dom S \subset \dom \overline{A_2}$; cf. \eqref{sss}. Moreover, it holds
  \begin{equation*}
    \| \psi_n^\lambda \| = \left| \left( \sqrt{\lambda^2 - m^2} \alpha_1 + m \beta + \lambda \right) \zeta \right|
      \cdot \left( \int_{B(0,1)} |\chi(|y|)|^2 \mathrm{d} y \right)^{1/2} = \text{const.},
  \end{equation*}
  and since the supports of the $\psi_n^\lambda$ are pairwise disjoint, the sequence $(\psi_n^\lambda)$
  converges weakly to zero. A straightforward  computation shows
  \begin{equation*}
    \begin{split}
      ( \overline{A_2} &- \lambda ) \psi_n^\lambda(x) = (S - \lambda ) \psi_n^\lambda(x) \\
      &= -\frac{i}{n^{5/2}} e^{i \sqrt{\lambda^2 - m^2} x \cdot e_1} 
          \chi'\left( \frac{1}{n} |x - x_n| \right) \alpha \cdot \frac{x - x_n}{|x-x_n|} 
          \left( \sqrt{\lambda^2 - m^2} \alpha_1 + m \beta + \lambda \right) \zeta \\
      &\qquad + \frac{1}{n^{3/2}} \chi\left( \frac{1}{n} |x - x_n| \right) e^{i \sqrt{\lambda^2 - m^2} x\cdot e_1} \\
      &\qquad \qquad \qquad \qquad \cdot \left( \sqrt{\lambda^2 - m^2} \alpha_1 + m \beta - \lambda \right)
          \left( \sqrt{\lambda^2 - m^2} \alpha_1 + m \beta + \lambda \right) \zeta.
    \end{split}
  \end{equation*}
  Note that $\left( \sqrt{\lambda^2 - m^2} \alpha_1 + m \beta - \lambda \right)
          \left( \sqrt{\lambda^2 - m^2} \alpha_1 + m \beta + \lambda \right) = 0$ by \eqref{equation_anti_commutation}.
          Hence, we have
  \begin{equation*}
    \big\| ( \overline{A_2} - \lambda) \psi_n^\lambda \big\| 
        \leq \frac{C}{n} \left( \int_{B(0,1)} |\chi'(|y|)|^2 \mathrm{d} y \right)^{1/2}
  \end{equation*}
  and therefore, $( \overline{A_2} - \lambda ) \psi_n^\lambda \rightarrow 0$. Thus
  $(\psi_n^\lambda)$ is a singular sequence for $\overline{A_2}$ and $\lambda$ and hence
  $\lambda\in \sigma_{\rm ess}(\overline{A_2})$.
  
  Assertions (ii) and (iv) follow from  \eqref{Birman_Schwinger_p}, \eqref{Krein_transformed}, and the 
  special form of the $\gamma$-field, Weyl function and $\Theta_{\rm max}(2)$ in \eqref{gammm}, \eqref{def_M_OBT}, and \eqref{def_theta_max}; 
  cf. also \cite[Corollary~3.14]{BM14}. 
  
  It remains to prove item~(iii). Since $(-\infty, -m] \cup [m, \infty) \subset \sigma_{\rm ess}(\overline{A_{\pm 2}})$ by~(i),
  it suffices to consider the case $\lambda \in (-m, m)$.
  Assume that $\lambda \in \sigma_{\rm disc}(\overline{A_2})$. 
  Note first that the Birman Schwinger principle \eqref{Birman_Schwinger_disc} implies 
  \begin{equation}\label{assu2}
   0 \in \sigma_{\rm disc}(\Theta_{\rm max}(2) - \mathcal{M}(\lambda)).
  \end{equation}
  A simple calculation using \eqref{equation_M_inv} shows 
  \begin{equation*} 
    -2 M(\lambda) \alpha \cdot \nu \left( \frac{1}{2} I_4 + M(\lambda) \right)
      = -\left( -\frac{1}{2} I_4 + M(\lambda) \right) \alpha \cdot \nu,
  \end{equation*}
  where the operators $M(\lambda)$ and $\alpha \cdot \nu$ are both bijective in $H^{1/2}(\Sigma; \mathbb{C}^4)$. Hence we have 
  \begin{equation} \label{equation_kernels}
    -\left( \frac{1}{2} I_4 + \widetilde M(\lambda) \right) \bigl(2 M(\lambda) \alpha \cdot \nu\bigr)'
      = -(\alpha \cdot \nu)'\left( -\frac{1}{2} I_4 + \widetilde M(\lambda) \right)
  \end{equation}
  with bijective operators $(2 M(\lambda) \alpha \cdot \nu)'$ and $(\alpha \cdot \nu)'$ in $H^{-1/2}(\Sigma; \mathbb{C}^4)$. 
  From \eqref{equation_kernels} 
  we conclude 
  \begin{equation}\label{kerne22}
    \dim \ker (\Theta_{\rm max}(-2) - \mathcal{M}(\lambda)) = \dim \ker(\Theta_{\rm max}(2) - \mathcal{M}(\lambda))
  \end{equation}
  and since $\ran (\Theta_{\rm max}(2) - \mathcal{M}(\lambda))$ is closed it follows from \eqref{equation_kernels} 
  that $\ran (\Theta_{\rm max}(-2) - \mathcal{M}(\lambda))$
  is closed; thus we have
   \begin{equation}\label{assu22}
   0 \in \sigma_{\rm disc}(\Theta_{\rm max}(-2) - \mathcal{M}(\lambda))
  \end{equation}
  and hence $\lambda \in \sigma_{\rm disc}(\overline{A_{-2}})$.
  In the same way one can show that $\lambda \in \sigma_{\rm disc}(\overline{A_{-2}})$
  implies $\lambda \in \sigma_{\rm disc}(\overline{A_2})$.
  
  Finally, we prove $\rho(\overline{A_2}) \cap (-m, m) = \rho(\overline{A_{-2}}) \cap (-m, m)$.
  By exclusion and the previous considerations this implies then
  $\sigma_{\rm ess}(\overline{A_2}) = \sigma_{\rm ess}(\overline{A_{-2}})$.
  Again, we only verify that $\rho(\overline{A_2}) \cap (-m, m) \subset \rho(\overline{A_{-2}}) \cap (-m, m)$,
  the other inclusion follows by symmetry.
  
  Let $\lambda \in \rho(\overline{A_2}) \cap (-m,m)$. Then, by~\eqref{Birman_Schwinger1} we have that 
  $0 \in \rho(\Theta_{\rm max}(2) - \mathcal{M}(\lambda))$.
  This implies that $\frac{1}{2} + \widetilde{M}(\lambda)$ is injective and that 
  $H^{1/2}(\Sigma; \mathbb{C}^4) \subset \ran\big( \frac{1}{2} + \widetilde{M}(\lambda) \big)$.
  Using equation~\eqref{equation_kernels} we deduce that $-\frac{1}{2} + \widetilde{M}(\lambda)$ is injective and that 
  $H^{1/2}(\Sigma; \mathbb{C}^4) \subset \ran\big( -\frac{1}{2} + \widetilde{M}(\lambda) \big)$,
  i.e. $0 \in \rho(\Theta_{\rm max}(-2) - \mathcal{M}(\lambda))$.
  Using again~\eqref{Birman_Schwinger1} we find $\lambda \in \rho(\overline{A_{-2}})$.
\end{proof}

\begin{remark}
  The functions $\psi_n^\lambda$ in \eqref{singular_sequence} are constructed as a solution of 
  the equation $(-i \alpha \cdot \nu + m \beta) f = 0$ times a cutoff function such that
  $\text{supp}\, \psi_n^\lambda \cap \Sigma = \emptyset$. Because of the last property
  $\psi_n^\lambda$ is contained in the domain of the symmetric operator $S$ in \eqref{sss}, and hence 
  $(\psi_n^\lambda)$ is a singular sequence for any self-adjoint extension of $S$ at $\lambda$. 
  This implies that the set $(-\infty, -m] \cup [m, \infty)$ is contained in the essential spectrum 
  of any self-adjoint extension of $S$; cf.~\cite[Theorem~4.4~(i)]{BEHL16_2}. 
\end{remark}

Note that Theorem~\ref{proposition_properties_A_pm2} does not state that the spectrum of 
$\overline{A_{\pm 2}}$ in $(-m,m)$ is purely discrete; in fact
essential spectrum may appear in the gap as well.
In the special case when the interaction support contains a flat part it turns out in the next theorem that 
$0 \in \sigma_{\rm ess}(\overline{A_{\pm 2}})$. Moreover, the functions in $\dom \overline{A_{\pm 2}}$ 
do not possess any Sobolev regularity (of positive order);
cf. Theorem~\ref{theorem_self_adjoint_bad}.

\begin{thm} \label{theorem_essential_spectrum}
  Let $\Sigma \subset \mathbb{R}^3$ be the boundary of a bounded $C^2$-smooth domain such that there 
  exists an open set $\Sigma_0 \subset \Sigma$ which is contained in a plane. 
  Then the following assertions hold for the self-adjoint operators $\overline{A_{\pm 2}}$: 
  \begin{itemize}
    \item[\rm (i)] $0 \in \sigma_{\rm ess}(\overline{A_{\pm 2}})$;
    \item[\rm (ii)] $\dom \overline{A_{\pm 2}} \not\subset H^s(\mathbb{R}^3 \setminus \Sigma; \mathbb{C}^4)$
  for all $s > 0$.
  \end{itemize}
\end{thm}
\begin{proof}
  (i) The proof of this item is shown in an indirect way and is split into four steps. 
  Again we restrict ourselves to the case $\eta=2$.
  Let us assume that 
  \begin{equation}\label{assu}
  0 \in \rho(\overline{A_2})\cup \sigma_{\rm disc}(\overline{A_2}) 
  \end{equation}
  and 
  consider the operator $\Xi:L^2(\Sigma; \mathbb{C}^4) \rightarrow L^2(\Sigma; \mathbb{C}^4)$
   defined by 
  \begin{equation} \label{def_Xi}
      \Xi \varphi := (I_4 - \Delta_\Sigma)^{1/4} 
        \left( \widetilde{M}(0)^2 - \frac{1}{4} \right) (I_4 - \Delta_\Sigma)^{1/4} \varphi,\quad \varphi\in L^2(\Sigma; \mathbb{C}^4).
  \end{equation}
  
  \noindent
  {\it Step 1.} Observe first that the operator $\Xi$ is bounded and self-adjoint in $L^2(\Sigma; \mathbb{C}^4)$. In fact, 
  $\widetilde{M}(0)^2 - \frac{1}{4} I_4: H^{-1/2}(\Sigma; \mathbb{C}^4) \rightarrow H^{1/2}(\Sigma; \mathbb{C}^4)$
  is bounded by Proposition~\ref{proposition_extension_gamma_Weyl}~(iii) and hence $\Xi$ is well defined
  and bounded in $L^2(\Sigma; \mathbb{C}^4)$.
  Moreover, since $M(0)$ is symmetric by \eqref{equation_diff_m} we have 
  \begin{equation*}
      \big( \Xi \varphi, \varphi \big) 
          = \left( \left( M(0)^2 - \frac{1}{4} \right) (I_4 - \Delta_\Sigma)^{1/4} \varphi, (I_4 - \Delta_\Sigma)^{1/4} \varphi \right)\in\mathbb{R}
  \end{equation*}
  for $\varphi\in H^1(\Sigma; \mathbb{C}^4)$. 
  By a density argument this extends to all $\varphi\in L^2(\Sigma; \mathbb{C}^4)$, so that $\Xi$ is self-adjoint in $L^2(\Sigma; \mathbb{C}^4)$.  
  
  \noindent
  {\it Step 2.} We claim that  the direct sum decomposition
  \begin{equation}\label{xikern}
   \ker \Xi = \ker \Theta_{\rm max}(2) \dot + \ker \Theta_{\rm max}(-2)
  \end{equation}
  holds. In particular, together with \eqref{kerne22} for $\lambda = 0$, $\mathcal{M}(0) = 0$ and assumption \eqref{assu} this implies that 
  $\dim\ker \Xi < \infty$.
  Note first that the sum in \eqref{xikern} is direct since 
  $\ker ( \frac{1}{2} + \widetilde{M}(0) ) \cap \ker ( \frac{1}{2} - \widetilde{M}(0)) = \{ 0 \}$.
  Next, the inclusion
  \begin{equation} \label{inclusion_kernels}
    \ker \Theta_{\rm max}(2) \dot + \ker \Theta_{\rm max}(-2) \subset \ker \Xi
  \end{equation}
  follows easily from
  \begin{equation}\label{gutzuhaben}
   \begin{split}
      \Xi &= (I_4 - \Delta_\Sigma)^{1/4} 
        \left( \frac{1}{2}+\widetilde{M}(0)\right) \left( -\frac{1}{2}+\widetilde{M}(0)\right)(I_4 - \Delta_\Sigma)^{1/4}\\
        &=(I_4 - \Delta_\Sigma)^{1/4} 
        \left( -\frac{1}{2}+\widetilde{M}(0)\right) \left( \frac{1}{2}+\widetilde{M}(0)\right)(I_4 - \Delta_\Sigma)^{1/4}.
        \end{split}
  \end{equation}
  Furthermore, \eqref{gutzuhaben} also yields
  \begin{equation*}
    \left( \widetilde{M}(0) - \frac{1}{2} \right) (I_4 - \Delta_\Sigma)^{1/4} \big(\ker \Xi \ominus \ker \Theta_{\rm max}(-2)\big)
    \subset \ker \left( \frac{1}{2} + \widetilde{M}(0) \right),
  \end{equation*}
  where $\ker \Xi \ominus \ker \Theta_{\rm max}(-2)$ denotes the orthogonal complement of $\ker \Theta_{\rm max}(-2)$
  in the closed subspace $\ker \Xi$ of $L^2(\Sigma; \mathbb{C}^4)$.
  Since the operator 
  $$( \widetilde{M}(0) - \frac{1}{2}) (I_4 - \Delta_\Sigma)^{1/4} 
  \upharpoonright \bigl(\ker \Xi \ominus \ker \Theta_{\rm max}(-2)\bigr)$$ 
  is injective and 
  $(I_4 - \Delta_\Sigma)^{1/4}$ is an isomorphism we find
  \begin{equation*}
    \begin{split}
      \dim \ker \Xi &\leq \dim \ker \left( \frac{1}{2} + \widetilde{M}(0) \right) 
          + \dim \ker \left( \frac{1}{2} - \widetilde{M}(0) \right) \\
      &= \dim \ker \Theta_{\rm max}(2) + \dim \ker \Theta_{\rm max}(-2),
    \end{split}
  \end{equation*}
  which together with \eqref{inclusion_kernels} implies \eqref{xikern}.   
  
   \noindent
  {\it Step 3.} Now we consider the restriction of the self-adjoint operator $\Xi$ onto the invariant subspace 
  $\mathscr{H} := (\ker\Xi)^\bot$.
  From the above considerations it is clear that $\Xi\!\upharpoonright_{\mathscr{H}}$
  is a bounded, self-adjoint and injective operator in $\mathscr{H}$.
  We claim that the operator $(\Xi\!\upharpoonright_{\mathscr{H}})^{-1}$ is bounded and everywhere defined 
  in $\mathscr{H}$.
  
  In the following let $P_\pm$ be the orthogonal projectors onto $\ker \Theta_{\rm max}(\pm 2)$
  and observe that the self-adjoint operators
  \begin{equation*}
   \Theta_{\rm max}(\pm 2)\upharpoonright_{(1-P_\pm) L^2(\Sigma; \mathbb{C}^4) }
  \end{equation*}
    are boundedly invertible in $(1-P_\pm) L^2(\Sigma; \mathbb{C}^4)$; this follows from~\eqref{assu},
    Theorem~\ref{proposition_properties_A_pm2}~(iii) and \eqref{Birman_Schwinger_disc}. 
    We shall denote these restrictions by $\Theta^\pm_{\rm max}(\pm 2)$. 
    Let $\varphi\in\ran\Xi \subset \mathscr{H}$ and choose $\psi\in\mathscr{H}$ 
    such that $\varphi=\Xi\psi$. It is easy to see that
    \begin{equation*}
     \psi_\pm:=- (I_4 - \Delta_\Sigma)^{-1/4} \left( \mp \frac{1}{2} + \widetilde{M}(0) \right) (I_4 - \Delta_\Sigma)^{1/4}
     \psi \in \dom \Theta_{\rm max}(\pm 2) 
    \end{equation*}
    satisfy $\varphi=\Xi\psi=\Theta_{\rm max}(\pm 2)\psi_\pm$. Then we have $\psi_\pm= \Theta^\pm_{\rm max}(\pm 2)^{-1}\varphi + P_\pm\psi_\pm$
and hence 
 \begin{equation}
  (\Xi\!\upharpoonright_{\mathscr{H}})^{-1}\varphi=\psi=\psi_+-\psi_-
  =\Theta^+_{\rm max}(2)^{-1}\varphi - \Theta^-_{\rm max}(- 2)^{-1}\varphi + P_+\psi_+ - P_-\psi_-.
 \end{equation}
Since $P_-\psi_- - P_+\psi_+ \in\ker\Xi=\mathscr{H}^\bot$ by \eqref{xikern} we find
\begin{equation*}
\begin{split}
 \Vert (\Xi\!\upharpoonright_{\mathscr{H}})^{-1}\varphi\Vert^2 & \leq \Vert (\Xi\!\upharpoonright_{\mathscr{H}})^{-1}\varphi\Vert^2 + \Vert P_-\psi_- - P_+\psi_+\Vert^2\\
 &=\bigl\Vert (\Xi\!\upharpoonright_{\mathscr{H}})^{-1}\varphi +  (P_-\psi_- - P_+\psi_+) \bigr\Vert^2\\
 &=\bigl\Vert \Theta^+_{\rm max}(2)^{-1}\varphi - \Theta^-_{\rm max}(- 2)^{-1}\varphi \bigr\Vert^2
 \end{split}
 \end{equation*}
and as $\Theta^\pm_{\rm max}(\pm 2)^{-1}$ are bounded it follows that 
$(\Xi\!\upharpoonright_{\mathscr{H}})^{-1}$ is bounded in $\mathscr{H}$. As $(\Xi\!\upharpoonright_{\mathscr{H}})^{-1}$ 
is self-adjoint in $\mathscr{H}$ it is clear that it is defined on $\mathscr{H}$.

  \noindent
  {\it Step 4.}
  Now we show that the assumption on $\Sigma_0 \subset \Sigma$ implies 
  that
  there are infinitely many linearly independent 
  functions which do not belong to the range of the operator $\Xi$ in \eqref{def_Xi}. 
  This is a contradiction to the fact that $\dim \ker\Xi $ is finite and 
  that $\Xi\!\upharpoonright_{\mathscr{H}}$ is boundedly invertible with an inverse defined on all of $\mathscr{H}$; thus \eqref{assu} can not be true. 
  
  As in the proof of Proposition~\ref{proposition_extension_gamma_Weyl}~(iii) consider
  \begin{equation*}
    \mathcal{A} = M(0) \alpha \cdot \nu + \alpha \cdot \nu M(0)
  \end{equation*}
  in $H^{1/2}(\Sigma; \mathbb{C}^4)$
  and recall that this operator admits a bounded extension
  $\widetilde{\mathcal{A}}: H^{-1/2}(\Sigma; \mathbb{C}^4) \rightarrow H^{1/2}(\Sigma; \mathbb{C}^4)$;
  cf. \cite[Proposition~2.8]{OV16} or~\eqref{equation_def_A} and the discussion afterwards. Since 
  $M(0)^2 - \frac{1}{4} I_4 = M(0) \alpha \cdot \nu \mathcal{A}$ 
  (see the proof of Proposition~\ref{proposition_extension_gamma_Weyl}~(iii))  a density argument leads to 
  \begin{equation*}
    \widetilde{M}(0)^2 - \frac{1}{4} I_4 
        = M(0) \alpha \cdot \nu \widetilde{\mathcal{A}}.
  \end{equation*}
  Since $M(0)$ and $\alpha \cdot \nu$ are bijective in $H^{1/2}(\Sigma; \mathbb{C}^4)$ and 
  $(I_4 - \Delta_\Sigma)^{1/4}$
  is an isomorphism, we see by comparing with \eqref{def_Xi} that infinitely many linearly independent functions do not belong to 
  $\ran \Xi$ if and only if infinitely many linearly independent functions do not belong to 
  $\ran \widetilde{\mathcal{A}}$. 
  This statement will be shown now.

  Making use of \eqref{equation_anti_commutation} we see that
  $\mathcal{A}$ is an integral operator of the form
  \begin{equation}\label{aaaaa}
    \mathcal{A} \varphi(x) = \int_\Sigma K(x, z) \varphi(z) \mathrm{d} \sigma(z)
  \end{equation}
  with integral kernel 
  \begin{equation*}
    K(x, z) = G_0(x-z)\alpha \cdot (\nu(z) - \nu(x)) + \frac{i e^{-m|x-z|}}{2 \pi |x-z|^3} (1 + m |x-z|) \nu(x)\cdot(x-z),
  \end{equation*}
  where $G_0$ is the Green's function for the resolvent of $A_0$ given by \eqref{def_G_lambda}. 
  Note that $\vert K(x,z)\vert\leq C\vert x-z\vert^{-1}$ and hence the integral operator in \eqref{aaaaa} is not singular 
  (see also \cite[equation~(22) and Lemma~3.5]{AMV14} and \cite[Proposition~3.11]{F95}). 
  Let $\Sigma_1 \subset \Sigma$ such that $\overline{\Sigma_1} \subset \Sigma_0$.
  Note that $K(x,z) = 0$, if $x, z \in \Sigma_0$.
  Let $U_1 \subset \mathbb{R}^2$ and $\phi: U_1 \rightarrow \mathbb{R}^3$ be a linear affine function 
  which parametrizes $\Sigma_1$,
  i.e. $\ran \phi = \Sigma_1$, and
  let $\varphi \in H^{1/2}(\Sigma; \mathbb{C}^4)$ be fixed.
  Since $\nu$ is constant on $\Sigma_0$ and $\overline{\Sigma_1} \subset \Sigma_0$, we see that the mapping 
  $U_1 \ni u \mapsto K(\phi(u), z)$ is $C^\infty$-smooth for any $z\in\Sigma$ and the mapping $\Sigma\ni z\mapsto  K(\phi(u), z)$ is $C^1$-smooth for any $u\in U_1$.
  From this, it is easy to deduce that $(\mathcal{A} \varphi) \circ \phi$ is differentiable on $U_1$ and
  \begin{equation*}
    \partial_{u_j} (\mathcal{A} \varphi)( \phi (u)) = \int_\Sigma \partial_{u_j} K(\phi(u), z) \varphi(z) \mathrm{d} \sigma(z),
    \qquad j \in \{ 1, 2 \}.
  \end{equation*}
  Let us denote the elements of the $4 \times 4$-matrix $K(x, z)$ by $K_{l m}(x, z)$ and the elements of
  $\varphi(x) \in \mathbb{C}^4$ by $\varphi_m(x)$, $l, m \in \{ 1, 2, 3, 4\}$. Then the last observation
  implies, in particular, that
  \begin{equation*}
    \begin{split}
      \| \partial_{u_j} \mathcal{A} \varphi \|_{L^2(\Sigma_1; \mathbb{C}^4)}^2
        &= C_1 \int_{U_1} \left| \partial_{u_j} \mathcal{A} \varphi(\phi(u)) \right|^2 \mathrm{d} u \\
      &= C_1 \int_{U_1} \left| \int_\Sigma \partial_{u_j} K(\phi(u), z) \varphi(z) \mathrm{d} \sigma(z) \right|^2 \mathrm{d} u \\
      &= C_1 \int_{U_1} \sum_{m, l = 1}^4 \left| \big( \partial_{u_j} K_{l m}(\phi(u), \cdot), \varphi_m \big)_{1/2 \times -1/2} \right|^2 \mathrm{d} u \\
      &\leq C_1 \int_{U_1} \big\| \partial_{u_j} K(\phi(u), \cdot)\big\|_{H^{1/2}(\Sigma; \mathbb{C}^4)}^2 
          \| \varphi \|_{H^{-1/2}(\Sigma; \mathbb{C}^4)} \mathrm{d} u \\
      &= C_2 \| \varphi \|_{H^{-1/2}(\Sigma; \mathbb{C}^4)}.
    \end{split}
  \end{equation*}
  By continuity we obtain from this observation that $\widetilde{\mathcal{A}} \varphi|_{\Sigma_1} \in H^1(\Sigma_1; \mathbb{C}^4)$
  for any $\varphi \in H^{-1/2}(\Sigma; \mathbb{C}^4)$.
  Thus, any $\psi \in H^{1/2}(\Sigma; \mathbb{C}^4)$ with $\psi|_{\Sigma_1} \notin H^1(\Sigma_1; \mathbb{C}^4)$
  is not contained in $\ran \mathcal{A}$. Hence, there are infinitely many linearly independent functions
  in $H^{1/2}(\Sigma; \mathbb{C}^4)$ that are not contained in $\ran \widetilde{\mathcal{A}}$. The proof of item (i) is complete.

  \vskip 0.2cm
  \noindent
  (ii) We show that
  $\dom \overline{A_{2}} \subset H^s(\mathbb{R}^3 \setminus \Sigma; \mathbb{C}^4)$
  for some $s > 0$ implies that the resolvent difference 
  $(\overline{A_{2}} - \lambda)^{-1}-(A_0 - \lambda)^{-1}$ is compact for $\lambda \in \mathbb{C} \setminus \mathbb{R}$.
  Since $\sigma_{\rm ess}(A_0)=(-\infty,-m]\cup[m,\infty)\not=\sigma_{\rm ess}(\overline{A_{2}})$ this is a contradiction.

  For $s \in [0, 1]$ consider the Hilbert spaces 
  \begin{equation*}
    \mathcal{H}^s := H^s(\mathbb{R}^3 \setminus \Sigma; \mathbb{C}^4) \cap \dom T_{\rm max}
  \end{equation*}
  equipped with the norms
  \begin{equation*}
    \| f \|_{\mathcal{H}^s}^2 := \| f \|_{H^s(\mathbb{R}^3 \setminus \Sigma; \mathbb{C}^4)}^2 
        + \| T_{\rm max} f\|_{L^2(\mathbb{R}^3; \mathbb{C}^4)}^2, \qquad f \in \mathcal{H}^s.
  \end{equation*}
  Then, the trace mappings 
  $\Gamma_j^1 := \Gamma_j: H^1(\mathbb{R}^3 \setminus \Sigma; \mathbb{C}^4) = \mathcal{H}^1
  \rightarrow H^{1/2}(\Sigma; \mathbb{C}^4)$ and 
  $\Gamma_j^0 := \widetilde{\Gamma}_j: \dom T_{\rm max} = \mathcal{H}^0 \rightarrow H^{-1/2}(\Sigma; \mathbb{C}^4)$
  are continuous for $j \in \{ 0, 1\}$. By interpolation  we get that also 
  \begin{equation*}
    \Gamma_j^s := \widetilde{\Gamma}_j \upharpoonright \mathcal{H}^s: \mathcal{H}^s \rightarrow H^{s-1/2}(\Sigma; \mathbb{C}^4)
  \end{equation*}
  is continuous for any $s \in [0, 1]$.
  
  Let us assume now that 
  $\dom \overline{A_{2}} = \ker \big( \Upsilon_1 - \Theta_{\rm max}(2) \Upsilon_0 \big) \subset \mathcal{H}^s$
  for some $s > 0$. Then we have  $\dom \Theta_{\rm max}(2) \subset H^s(\Sigma; \mathbb{C}^4)$ 
  as $\Upsilon_0 = (I_4 - \Delta_\Sigma)^{-1/4} \widetilde{\Gamma}_0$.
  Let $\beta$ and $\mathcal M$ be the $\gamma$-field and Weyl function  
  corresponding to
  $\{ L^2(\Sigma; \mathbb{C}^4), \Upsilon_0, \Upsilon_1 \}$; cf. \eqref{gammm} and \eqref{def_M_OBT}.
  For $\lambda\in\mathbb C\setminus\mathbb R$ we have  
  \begin{equation*}
    \ran \big( \Theta_{\max}(2) - \mathcal{M}(\lambda) \big)^{-1}
        = \dom \big( \Theta_{\max}(2) - \mathcal{M}(\lambda) \big)
        \subset H^s(\Sigma; \mathbb{C}^4)
  \end{equation*}
  and  $\big( \Theta_{\max}(2) - \mathcal{M}(\lambda) \big)^{-1}$ is continuous in 
  $L^2(\Sigma; \mathbb{C}^4)$. It follows that the operator
  \begin{equation*}
    \big( \Theta_{\max}(2) - \mathcal{M}(\lambda) \big)^{-1}: L^2(\Sigma; \mathbb{C}^4) \rightarrow H^s(\Sigma; \mathbb{C}^4)
  \end{equation*}
  is closed and hence 
  continuous. As the embedding 
  $\iota_s: H^s(\Sigma; \mathbb{C}^4) \rightarrow L^2(\Sigma; \mathbb{C}^4)$ is compact we conclude that 
  $(\Theta_{\max}(2) - \mathcal{M}(\lambda))^{-1}$ is a compact operator
  in $L^2(\Sigma; \mathbb{C}^4)$. Eventually \eqref{Krein_transformed} yields that
  \begin{equation*}
      (\overline{A_{2}} - \lambda)^{-1} 
        - (A_0 - \lambda)^{-1} = \beta(\lambda) \big( \Theta_{\max}(2) - \mathcal{M}(\lambda) \big)^{-1}
        \beta(\overline{\lambda})^* ,\quad \lambda\in\mathbb C\setminus\mathbb R,
  \end{equation*}
  is a compact operator in $L^2(\mathbb{R}^3; \mathbb{C}^4)$. This completes the proof.
\end{proof}

\end{document}